\documentclass[aos]{imsart}
\usepackage{amsthm}
\usepackage{amsmath}
\usepackage{bbm}
\usepackage{epsfig}
\usepackage{graphicx}
\usepackage{mathrsfs}
\usepackage{amsthm}
\usepackage{enumerate}
\usepackage{amsfonts}
\usepackage{color}
\usepackage{rotating}
\usepackage{verbatim}
\usepackage{hyperref}
\usepackage[authoryear,round,comma]{natbib}
\usepackage{amsfonts}
\usepackage{amssymb}
\usepackage{appendix}
\usepackage{amsthm}
\usepackage{footnote}

\DeclareGraphicsExtensions{.eps,.png,.jpg,.pdf}

\newtheorem{definition}{Definition}[section]

\newtheorem{lemma}{Lemma}[section]
\newtheorem{proposition}{Proposition}[section]
\newtheorem{theorem}{Theorem}[section]
\newtheorem{corollary}{Corollary}[section]

\theoremstyle{remark}
\newtheorem{remark}{\bf{Remark}}[section]
\newtheorem{example}{\bf{Example}}[section]

\newcommand{\ConvD}{\overset{d}{\rightarrow}}
\newcommand{\ConvP}{\overset{p}{\rightarrow}}

\newcommand{\Cov}{\mathrm{Cov}}
\newcommand{\Corr}{\mathrm{Corr}}
\newcommand{\Var}{\mathrm{Var}}
\newcommand{\E}{\mathbb{E}}

\makeatletter
\def\blfootnote{\xdef\@thefnmark{}\@footnotetext}
\makeatother

\begin{document}
\begin{frontmatter}
\title{On the validity of resampling methods under long memory}
\runtitle{Validity of resampling under long memory}



\begin{aug}
\author{\fnms{Shuyang} \snm{Bai}\thanksref{m1}\ead[label=e1]{bsy9142@uga.edu}} and
\author{\fnms{Murad} S. \snm{Taqqu}\corref{tq}\thanksref{t1,m2}\ead[label=e2]{murad@bu.edu}}

\thankstext{t1}{Corresponding author}
\runauthor{Bai, S.  and  Taqqu, M.S.}

\affiliation{University of Georgia\thanksmark{m1}, Boston University\thanksmark{m2}}

\address{Shuyang Bai\\
Department of Statistics\\
University of Georgia\\
101 Cedar Street\\
Athens, Georgia, 30606, US\\
\printead{e1}
}
\address{Murad S. Taqqu\\
Department of Mathematics and Statistics\\
Boston University\\
111 Cumminton Mall\\
Boston, MA, 02215, US\\
\printead{e2}
}
\end{aug}

\begin{abstract}
For long-memory time series, inference based on resampling is of crucial importance, since the asymptotic distribution can often be non-Gaussian and is difficult to determine statistically. However due to the strong dependence, establishing the asymptotic validity of resampling methods is nontrivial.  In this paper, we derive an efficient bound for the canonical correlation between two finite blocks of a  long-memory time series.  We show how this bound can be applied to establish the asymptotic consistency of subsampling procedures for general statistics under long memory. It allows the subsample size $b$ to be  $o(n)$,  where $n$ is the  sample size,  irrespective of the strength of the memory.  We are then able to improve many results found in the literature. We also consider applications of  subsampling procedures under long memory to the sample covariance, M-estimation and empirical processes.

\end{abstract}



\end{frontmatter}

\section{Introduction}

A stationary time series $\{X_n\}$ is said to have ``\emph{long memory}'', also called ``\emph{long-range dependence}'', if the covariance $\Cov[X_n,X_0]$ decays slowly like $n^{2d-1}$ as $n\rightarrow\infty$, where $2d-1\in (-1,0)$. The parameter $$
d\in (0,1/2)
 $$
 is called the \emph{memory parameter}. Time series exhibiting long memory has been found frequently in practice.  Statistical  problems under such a context have been widely studied.  We refer the reader  to the recent monographs \citet{doukhan:oppenheim:taqqu:2003:theory}, \citet{giraitis:koul:surgailis:2009:large} and \citet{beran:2013:long} and \citet{pipiras:taqqu:2017:long} for more information.

Long memory creates a challenge for  large-sample inference. This is because the distributional scaling limits of some common statistical functionals, e.g. sample sum and quadratic forms, may be  non-Gaussian distributions due to the so-called \emph{non-central limit theorems}  (see, e.g., \citet{dobrushin:1979:gaussian}, \citet{taqqu:1979:convergence} and  \citet{terrin:taqqu:1990:noncentral}).  This typically leads to the following situation: even though all possible asymptotic distributions of a statistic  may have been derived,   it is often difficult to determine in practice which  the relevant one is based on the observations.  In such a situation, it is natural to resort to resampling.

A common strategy for resampling dependent data  is the so-called ``{\bf moving block bootstrap}''. Suppose that a stationary time series $\{X_1,\ldots,X_n\}$ is observed. The moving block bootstrap  performs the following procedure:

 1) with $n-b+1$ consecutive blocks of size $b$:
 $$
 \{X_1,\ldots,X_b\},\ \{X_2,\ldots,X_{b+1}\},\ \ldots \ , \ \{X_{n-b+1},\ldots,X_n\},
  $$
  sample randomly about $n/b$ blocks and paste them together to get a bootstrapped copy of $\{X_1,\ldots,X_n\}$;

 2) Compute the  statistics of interest on this bootstrapped copy;

  3)  Repeat the preceding two steps many times to get an empirical distribution of  the statistics for inference.

   Note that the moving block bootstrap involves rearranging the order of the time series and this can destroy the dependence structure.

  An alternative resampling scheme, called ``{\bf subsampling}'', directly computes the statistics on the block subsamples:  $\{X_1,\ldots,X_b\}$, $\{X_2,\ldots,X_{b+1}\},\ldots, \{X_{n-b+1},\ldots,X_n\}$,  and uses the resulting empirical distribution for inference. It usually involves a proper rescaling since the sample size has been reduced from $n$ to $b$.  Note that in subsampling, the original  time order is intact.
Subsampling appears to be a more robust procedure when there is strong dependence. Indeed,
\citet{lahiri:1993:moving} noticed that the moving block bootstrap procedure may fail in the long-memory case, while the subsampling method is shown to work  at least in the special case of sample mean (see  \citet{hall:1998:sampling}, \cite{nordman:2005:validity} and \cite{zhang:2013:block}).
For general information on resampling dependent data, see the monographs \citet{politis:1999:subsampling}, \citet{lahiri:2003:resampling} and Chapter 10 of \citet{beran:2013:long}.

The asymptotic validity of a subsampling procedure is usually formulated under the following setup: as the sample size $n\rightarrow\infty$, the block (subsample) size $b=b_n\rightarrow\infty$, while $b_n$ grows more slowly than the sample size:
\begin{equation*}
b_n=o(n).
\end{equation*}
 \citet{politis:romano:1994:large} established the validity of subsampling for general stationary time series by imposing an implicit  condition, which we call the ``subsampling condition'':
\begin{equation}\label{eq:block condition alpha}
\sum_{k=1}^n \alpha_{k,b_n}=o(n),  \quad\text{as }n\rightarrow\infty,
\end{equation}
where $\alpha_{k,b}$ is the between-block mixing coefficient defined by
\begin{equation}\label{eq:alpha_{k,b}}
\alpha_{k,b}=\sup\{|P(A\cap B)-P(A)P(B)|, ~A\in \mathcal{F}_1^b, ~B\in \mathcal{F}_{k+1}^{k+b}\},
\end{equation}
 with $\mathcal{F}_{p}^q$ denoting the sigma field generated by $\{X_{p},\ldots,X_{q}\}$. The  condition (\ref{eq:block condition alpha}) will be discussed in more details in Section \ref{sec:app subsampling} below. It is, however, not clear whether the natural block size condition $b_n=o(n)$ implies the subsampling condition (\ref{eq:block condition alpha}) in general. Under a typical weak dependence condition, strong mixing (see \citet{bradley:2007:mixing}), where it is assumed that as $k\rightarrow\infty$,
\begin{equation}\label{eq:strong mixing}
\alpha_{k}:=\sup\{|P(A\cap B)-P(A)P(B)|, ~A\in \mathcal{F}_{-\infty}^0,~ B\in \mathcal{F}_k^{\infty}\}\rightarrow 0,
\end{equation}
one can deduce easily that (\ref{eq:block condition alpha}) holds if  $b_n=o(n)$. See  \citet{politis:1999:subsampling}, Theorem 4.2.1. Allowing $b_n=o(n)$ has  important statistical implications. See Remark \ref{Rem:b=o(n) important}  below.

In the case where $\{X_n\}$ is given by a long-memory model,  the implication $b_n=o(n)$ $\Rightarrow$ (\ref{eq:block condition alpha})   has not been established as far as we know. In this paper, for a common class of long-memory models (see (\ref{eq:subordination}) below), we shall derive  an efficient bound on the between-block mixing coefficient $\alpha_{k,b}$ in (\ref{eq:alpha_{k,b}}). Such a bound entails that $b_n=o(n)$ implies (\ref{eq:block condition alpha}).

The paper is organized as follows.  In Section \ref{sec:gaussian}, we introduce the long-memory Gaussian subordination model  and discuss its important properties related to $\alpha_{k,b}$  in (\ref{eq:alpha_{k,b}}).
In Section \ref{sec:main results}, we state the main results. In Section \ref{sec:app subsampling}, the implication of the results on subsampling of long-memory Gaussian subordination model is discussed. Many results in the literature are improved. We consider the subsampling of some common statistics in Section \ref{sec:other common stat}.  Proofs of the main results in Section \ref{sec:main results} are given in Section \ref{sec:proofs}.

\section{Gaussian subordination and canonical correlation}\label{sec:gaussian}
In this paper, we focus on a typical class of long memory models, which is obtained by the so-called \emph{Gaussian subordination}. Let $\{Z_n\}$ be  stationary Gaussian process with   covariance function $\gamma(n)=\Cov[Z_n,Z_0]$. Suppose $\E Z_n=0$ and $\Var [Z_n]=1$.
 Write
\begin{equation}\label{e:VZ}
\mathbf{Z}_p^q= \left(Z_p,\ldots, Z_q \right).
\end{equation}
Denote the covariance matrix of $\mathbf{Z}_1^b$ by
\begin{equation}\label{eq:Sigma_m}
\Sigma_b=\Big(\gamma(i_2-i_1)= \Cov[Z_{i_1},Z_{i_2}]\Big)_{1\le i_1,i_2\le b} .
\end{equation}
We assume throughout that $\Sigma_b$ is non-singular for every $b\in \mathbb{Z}_+$.
Consider two vectors
$$\mathbf{Z}_{1}^b \quad \mbox{\rm and} \quad \mathbf{Z}_{k+1}^{k+b}
 $$
 distant by $k\in \mathbb{Z}_+$ of the same dimension $b$.
 Denote
the cross-block covariance matrix between $\mathbf{Z}_{1}^b$ and $\mathbf{Z}_{k+1}^{k+b}$ by
\begin{equation}\label{eq:Sigma_k,b}
\Sigma_{k,b}=\Big(\gamma(i_2+k-i_1)=   \Cov[Z_{i_1},Z_{i_2+k}]\Big)_{1\le i_1,i_2\le b}.
\end{equation}

The \emph{canonical correlation} $\rho(\cdot,\cdot)$ between two blocks of $\{Z_n\}$ of size $b$ differing by a translation of $k$ units in time is defined as
\begin{equation}\label{eq:rho_k,b}
\rho_{k,b}= \rho\left(\mathbf{Z}_1^b , \mathbf{Z}_{k+1}^{k+b}\right):=
\sup_{\mathbf{u}\in \mathbb{R}^b, \mathbf{v}\in \mathbb{R}^b} \Corr\Big(\langle\mathbf{u},\mathbf{Z}_1^b\rangle, \langle\mathbf{v}, \mathbf{Z}_{k+1}^{k+b}\rangle\Big)=\sup_{\mathbf{u}, \mathbf{v}\in \mathbb{R}^{b}}\frac{\mathbf{u}^T \Sigma_{k,b} \mathbf{v}}{\sqrt{\mathbf{u}^T \Sigma_b \mathbf{u}} \sqrt{\mathbf{v}^T \Sigma_b \mathbf{v}}},
\end{equation}
where $\langle \cdot, \cdot\rangle$ denote the Euclidean inner product.
Note that
\begin{equation}\label{eq:rho_k,b<=1}
0\le \rho_{k,b} \le 1.
\end{equation}

Consider now the observed stationary time series  $\{X_n\}$. The so-called \emph{Gaussian subordination} model is given by
\begin{equation}\label{eq:subordination}
X_n=G(Z_n),
\end{equation}
where $G(\cdot)$ is a measurable function.   We then say that $\{X_n\}$ is \emph{subordinated to} the underlying Gaussian $\{Z_n\}$. In fact, our results incorporate also the more general case
\[X_n=G(Z_n,\ldots,Z_{n-l}),
\]
 where $l$ is a fixed  non-negative integer. See, e.g., \citet{bai:taqqu:zhang:2015:unified}. We focus for simplicity only on the instantaneous case $l=0$, as is mostly considered in the literature. We note the following important quantity: if $\E G(Z)^2<\infty$ for a standard Gaussian $Z$, then the \emph{Hermite rank} $m$ of $G(\cdot)$ is defined as
\begin{equation}\label{eq:Hermite Rank Def}
m=\inf\{p\in \{1,2,3,\ldots\}: ~\E\big[ (G(Z)-\E G(Z)) Z^p\big]\neq 0\}.
\end{equation}
An equivalent definition with the monomial $Z^p$ replaced by $p$-th order Hermite polynomial is often used as well (see Definition 3.2 in \citet{beran:2013:long}).

Suppose now that $\{Z_n\}$ has long memory with memory parameter $d\in (0,1/2)$, that is in our context,
$$
\Cov[Z_k,Z_0]\sim c k^{2d-1}
 $$
 as $k\rightarrow\infty$, where $c$ is a positive constant, so that
 $$\sum_{k=-\infty}^\infty|\Cov[Z_k,Z_0]|=\infty.
 $$
     Then the covariance $\Cov[X_k,X_0]$  behaves like $k^{(2d-1)m}$ as $k\rightarrow\infty$. Hence  if $(2d-1)m>-1$, then $\{X_n\}$ has long memory;  if $(2d-1)m<-1$, then $\{X_n\}$ does not have long memory. See \citet{dobrushin:major:1979:non} and \citet{taqqu:1979:convergence} for more details.

The model (\ref{eq:subordination}) has some important mathematical advantages. First, it allows  various limit theorems involving short or long memory as well as light tails, namely finite variance (\citet{dobrushin:1979:gaussian}, \citet{taqqu:1979:convergence}, \citet{breuer:major:1983:central}),  but also heavy tails (\citet{sly:heyde:2008:nonstandard}). The second advantage, which greatly simplifies the analysis of the between-block mixing coefficient $\alpha_{k,b}$ in (\ref{eq:alpha_{k,b}}), is the following fact (Theorem 1 of \citet{kolmogorov:razanov:1960:strong}):

 For any jointly Gaussian vectors $(\mathbf{Z}_1,\mathbf{Z}_2)$, one has
\begin{equation}\label{eq:Kol-Roz}
\sup_{F,G}\left|\Corr\big(F_1(\mathbf{Z}_1),F_2(\mathbf{Z}_2)\big)\right|=  \rho\left(\mathbf{Z}_1 , \mathbf{Z}_{2} \right),
\end{equation}
where $\rho(\cdot,\cdot)$ is the canonical correlation as in (\ref{eq:rho_k,b}), and the supremum is taken over all functions $F_1,F_2$ such that $\E F_1(\mathbf{Z}_1)^2<\infty$, $\E F_2(\mathbf{Z}_2)^2<\infty$.

 Note that (\ref{eq:Kol-Roz}) reduces nonlinear functions $F_1$ and $F_2$ to linear functions.
Let $\alpha_{k,b}$ be the between-block mixing coefficient of $\{X_n\}$ defined in (\ref{eq:alpha_{k,b}}).
In view of (\ref{eq:Kol-Roz}), one has (see  Theorem 2 of \citet{kolmogorov:razanov:1960:strong}):
\begin{equation}\label{eq:alpha<=rho}
\alpha_{k,b}\le \rho_{k,b}.
\end{equation}
Hence  the subsampling condition (\ref{eq:block condition alpha})  holds if
\begin{equation}\label{eq:block cond gen}
\sum_{k=1}^{n} \rho_{k,b_n}=o(n),    \quad \text{as }n\rightarrow\infty.
\end{equation}
In the context of Gaussian subordination,  we shall also call (\ref{eq:block cond gen})  the \emph{subsampling condition}.

There has been some recent progress on deriving  (\ref{eq:block cond gen}) from some growth conditions more stringent than $b_n=o(n)$.
\citet{bai:taqqu:zhang:2015:unified} proved the following  bound on the canonical correlation $\rho_{k,b}$ (see  Lemma 3.4  of \citet{bai:taqqu:zhang:2015:unified}).
\begin{proposition}\label{Pro:crude}
Let
\begin{equation}\label{e:Mmax}
M_\gamma(k)=\max_{n> k} |\gamma(n)|,
\end{equation}
and let
\begin{equation}\label{eq:lambda min}
\text{$\lambda_b=$ the {\it minimum} eigenvalue of  $\Sigma_b$}.
\end{equation}
Then
\begin{equation}\label{eq:crude bound rho_k,b}
\rho_{k,b}\le  b \frac{M_\gamma(k-b)}{\lambda_b}.
\end{equation}
\end{proposition}
The bound (\ref{eq:crude bound rho_k,b}) is obtained  by optimizing the denominator and numerator in (\ref{eq:rho_k,b}) separately. Note that the minimum eigenvalue $\lambda_b$ is positive since $\Sigma_b$ is assumed to be positive definite.

How useful is the bound (\ref{eq:crude bound rho_k,b})? As shown in \citet{bai:taqqu:zhang:2015:unified}, this bound is satisfactory in cases where $\{Z_n\}$ has weak dependence. For example, suppose that $\gamma(n)$, $n\ge 1$, is bounded by  $cn^{-\beta}$ with $\beta>1$ and $c>0$,  and that the spectral density is bounded away from zero.  Then
$$
\sum_{k=1}^\infty M_\gamma(k)<\infty,\quad \mbox{\rm and} \quad \lambda_{\min}:=\inf_m\lambda_m>0
$$ (\citet{brockwell:1991:time}, Proposition 4.5.3) . Consequently if $b_n=o(n)$, we have
$$
\sum_{k=1}^n \rho_{k,b_n}\le b_n \lambda_{\min}^{-1}\sum_{k=1}^\infty  M_\gamma(k) =o(n).
 $$
 Thus $b_n=o(n)$ implies (\ref{eq:block cond gen}).

  However,
in the case where $\{Z_n\}$ has long memory with a memory parameter  $d\in (0,1/2)$, the crude bound (\ref{eq:crude bound rho_k,b}) is not satisfactory as it requires a block size
\begin{equation}\label{eq:b_n naive restr}
b_n=o(n^{1-2d})
\end{equation}
to obtain the subsampling condition (\ref{eq:block cond gen}). See the relation (57) in \citet{bai:taqqu:zhang:2015:unified}. The block size condition (\ref{eq:b_n naive restr}) depends on $d$ and becomes quite restrictive when $d$ is close to $1/2$.

Based on the result of \citet{adenstedt:1974:large},   \citet{betken:Wendler:2015:subsampling} have  obtained recently  a bound which improves  (\ref{eq:crude bound rho_k,b}) in the long memory situation.  Under the assumptions given in the Appendix,   \citet{betken:Wendler:2015:subsampling} derived the following bound for $\rho_{k,b}$:

\begin{proposition}\label{Pro:bw}
Let $\rho_{k,b}$ be as in (\ref{eq:rho_k,b}). Suppose that the assumptions BW1 and BW2 in the Appendix hold, where the covariance is
$$
\gamma(n)=n^{2d-1}L_\gamma(n)
$$
 for some slowly varying $L_\gamma$, $d\in (0,1/2)$. Then when $k>b$, we have  for some constants $C_1,C_2>0$ that
\begin{equation}\label{eq:BW ineq}
\rho_{k,b}\le
 C_1 \left(\frac{b}{k-b}\right)^{1-2d}
 L_\gamma(k-b)
 +C_2 b^2 (k-b)^{2d-2}\max\{L_\gamma(k-b), 1\}.
\end{equation}
\end{proposition}
In order to ensure   (\ref{eq:block cond gen}) using (\ref{eq:BW ineq}),  it is imposed in \citet{betken:Wendler:2015:subsampling} that
\begin{equation}\label{eq:BW block}
b_n=o(n^{1-d-\epsilon})
\end{equation}
for some arbitrarily small $\epsilon>0$. Note that the requirement (\ref{eq:BW block}), although it depends on $d$, is less restrictive than (\ref{eq:b_n naive restr}), and the order $b_n=O(n^{1/2})$ is always allowed since $d<1/2$.
The right-hand side of (\ref{eq:BW ineq}) involves two terms. If it were not for the presence of the second term,  one would be able to impose only $b_n=o(n)$.
\begin{remark}\label{Rem:b=o(n) important}
Allowing the non-restrictive condition $b_n=o(n)$ has an important practical implication: the valid range of block size choice does not depend on the memory parameter $d$ of $\{Z_n\}$, which is typically unknown. This is particularly desirable for resampling procedures designed to avoid treating the nuisance parameter $d$ under long memory (see, e.g., \citet{jach:2012:subsampling} and \citet{bai:taqqu:zhang:2015:unified}).
\end{remark}

In Section \ref{sec:main results} below,  we shall  bound  $\rho_{k,b}$ in such a way that   only the first term in (\ref{eq:BW ineq}) is effectively present for a wide class of long memory models. This  will allow  the nonrestrictive condition  $b_n=o(n)$ to imply the subsampling condition (\ref{eq:block cond gen}) (see Theorem \ref{Thm:block cond} below). We obtain our result by  establishing a tight bound in the special case of a $\mathrm{FARIMA}(0,d,0)$ time series, where there is   explicit information  on the model, and then extending the result to a more general setup using  Fourier analysis.

We  mention that our results  can be easily generalized to obtain a bound on the canonical correlation $\rho\left(\mathbf{Z}_1^{a},\mathbf{Z}_{k+1}^{k+b}\right)$, where possibly different block sizes $a$ and $b$ are involved. Since the main application,  subsampling,   involves only  $a=b$,  the proofs  are given only for this case. Nevertheless,  the statements concerning the case $a\neq b$ are given in Remark (\ref{Rem:diff m}) below.

\section{Main results}\label{sec:main results}
In this section we state the main result.
Let $\{Z_n\}$ be a second-order stationary process with covariance function $\gamma(n)$ and spectral density $f(\lambda)$ so that
\[
\gamma(n)=\int_{-\pi}^{\pi} f(\lambda) e^{in\lambda}d\lambda.
\]
Let $\rho_{k,b}$ be as in (\ref{eq:rho_k,b}).
Consider the spectral density of a $\mathrm{FARIMA}(0,d,0)$ time series  (see, e.g., Section 13.2 of \citet{brockwell:1991:time}):
\begin{equation}\label{eq:f_d}
f_d(\lambda):=\frac{1}{2\pi} \left|1-e^{i\lambda}
\right|^{-2d}=\frac{1}{2\pi} \left[\sin(\lambda/2)^2\right]^{-d},\quad 0<d<1/2.
\end{equation}
First we state a result in the special case of $\mathrm{FARIMA}(0,d,0)$ time series.
\begin{theorem}\label{Thm:special}
Suppose that $f(\lambda)=\sigma^2f_d(\lambda)$, for some constant $\sigma^2>0$,   is the spectral spectral density of $\{Z_n\}$. Then when $ 1\le  b<k$, there exists a constant  $c>0$, such that
\begin{equation}\label{eq:bound farima}
\rho_{k,b}\le c \left(\frac{b}{k-b}\right)^{1-2d}.
\end{equation}
\end{theorem}
\noindent This theorem is proved in Section \ref{sec:pf thm}.
\begin{remark}
We note that the bound (\ref{eq:bound farima}) is  sharp when $b\ll k$.  Indeed, suppose that $\gamma_d(n)$ is the covariance of the FARIMA$(0,d,0)$ (see (\ref{eq:gamma_d}) bellow). $\gamma_d(n)$ is decreasing in $n\in\mathbb{Z}_+$ following the asymptotic order $n^{2d-1}$. Then it is well-known  that (see e.g., Corollary 1.2 of \cite{beran:2013:long})
\[
V_d(b):= \Var[Z_1+\ldots+Z_b]  \sim c b^{2d+1}
\]
for some constant $c>0$.   Now take in (\ref{eq:rho_k,b})  $\mathbf{u}=\mathbf{v}=(1,\ldots,1)^T$.  One has for $k>b\ge 1$,
\begin{align}
\rho_{k,b}&\ge \mathrm{Corr}(Z_1+\ldots+Z_b, Z_{k+1}+\ldots+ Z_{k+b})=  V_d(b)^{-1} \sum_{i,j=1}^b\gamma_d(i-j+k)   \notag
\\&\ge c_1 b^{-2d-1} \sum_{i,j=1}^b\gamma_d(b-1+k)  \ge c_2 b^{-2d-1} b^2 (k+b-1)^{2d-1}=c_2    \left(\frac{b}{k+b-1}\right)^{1-2d},\label{eq:sharp}
\end{align}
for some constants $c_1,c_2>0$, where the first inequality in (\ref{eq:sharp}) follows from the fact that $\gamma_d(n)$ is decreasing in $n$, and $\gamma_d(i-j+k)$ attains its minimum at $i=b,j=1$.  When $b\ll k$ (e.g, when $k>2b$),  the bounds (\ref{eq:bound farima}) and (\ref{eq:sharp}) are both following the order $(b/k)^{1-2d}$.
\end{remark}

\begin{remark}\footnote{The remark is suggested by an anonymous referee.}
It is instructive to relate the bound (\ref{eq:bound farima}) to strong mixing, a weak dependence notion in the case where $\{Z_n\}$ is a stationary Gaussian sequence.
 By \citet{kolmogorov:razanov:1960:strong}, the Gaussian $\{Z_n\}$ is strong mixing\footnote{
For stationary Gaussian processes, the ``complete regularity coefficient'' $\rho_{k,b}$ is expressed as (\ref{eq:rho_k,b}). It is equivalent to $\alpha_{k,b}$ in (\ref{eq:alpha_{k,b}}), that is, $\alpha_{k,b}\le \rho_{k,b}\le  2\pi \alpha_{k,b}$. Hence a stationary Gaussian process is strong mixing if and only if (\ref{eq:strong mix Gaussian}) holds. See  \citet{kolmogorov:razanov:1960:strong} Theorem 1 and 2 or \citet{ibragimov:rozanov:1978:gaussian}, Section IV.1.  See also \citet{bai:taqqu:zhang:2015:unified}.} if and only if
\begin{equation}\label{eq:strong mix Gaussian}
\lim_{k\rightarrow\infty}\sup_{b\in \mathbb{Z}_+}\rho_{k+b,b}=0.
\end{equation}
Since the long memory $\{Z_n\}$ with a spectral density in (\ref{eq:f_d}) is not strong mixing, we must have
\begin{equation}\label{eq:non strong mixing Gaussian}
\lim_{k\rightarrow\infty}\sup_{b\in \mathbb{Z}_+}\rho_{k+b,b}>0.
\end{equation}
Let us then rewrite (\ref{eq:bound farima}) as
\begin{equation}\label{e:rhokb}
\rho_{k+b,b}\le c \left(\frac{b}{k}\right)^{1-2d}.
\end{equation}
The bound (\ref{e:rhokb}) shows that  with a fixed $b$, we have
\[
\rho_{k+b,b}=O(k^{2d-1}),
\]
and hence
$$
\lim_{k\rightarrow\infty}\rho_{k+b,b}=0.
$$
This   relation is weaker than (\ref{eq:strong mix Gaussian}), and hence in accord with the fact that   $\mathrm{FARIMA}(0,d,0)$ with $0<d<1/2$ is not strong mixing.

We also note that  it follows from Lemma \ref{Lem:reduction max} that a reversed version of (\ref{eq:strong mix Gaussian}):
\[
\lim_{b_n\rightarrow\infty} \sup_{[\epsilon n]\le k\le n} \rho_{k,b_n}=0, \quad   b_n=o(n),
\]
where $\epsilon>0$ is arbitrarily small,  is   a sufficient condition for the subsampling condition (\ref{eq:block cond gen}).
\end{remark}

\medskip
We now turn to the general case:
\begin{theorem}\label{Thm:main}
Suppose that the spectral density of $\{Z_n\}$ is given by
\begin{equation}\label{eq:f=f_df_0}
f(\lambda)=f_d(\lambda)f_0(\lambda),
\end{equation}
where $f_0(\lambda)$  is the spectral density which corresponds to a covariance function (or  Fourier coefficient)
\[\gamma_0(n)=\int_{-\pi}^{\pi} f_0(\lambda) e^{in\lambda}d\lambda. \] Assume that the following holds:
\begin{enumerate}[(a)]
\item $c_0:=\inf_{\lambda\in (-\pi,\pi]}f_0(\lambda)>0$;
\item $\sum_{n=-\infty}^\infty|\gamma_0(n)|<\infty$.
\end{enumerate}
Then, depending on the rate of decay of $\gamma_0(n)$,  we have the following:
\begin{itemize}
\item Suppose
$\gamma_0(n) = O(n^{-\alpha})$  for some $\alpha> 0$: \\ \phantom{a}
Then for any fixed small $\epsilon>0$,  there exist constants $c_1,c_2>0$,  such that for all $b,k',k$ satisfying
\begin{equation}\label{eq:m k' restrict}
1 \le   b<k'\le (1-\epsilon)k,
\end{equation}
we have
\begin{equation}\label{eq:rho_k,b bound main}
\rho_{k,b}\le c_1 \left(\frac{b}{k'-b}\right)^{1-2d} + c_2 {k'}{k^{-\alpha}}.
\end{equation}
\item Suppose $\gamma_0(n)=o(n^{-\alpha})$  for some $\alpha> 0$: \\  \phantom{a}
Then the second term  on the right-hand side of (\ref{eq:rho_k,b bound main}) can be replaced by $c_2k'o(k^{-\alpha})$.
\item Suppose $\gamma_0(n)=O(e^{-cn})$ for some $c>0$: \\ \phantom{a}
Then the second term  on the right-hand side of (\ref{eq:rho_k,b bound main}) can be replaced by $c_2 e^{-c_3 k}$ for some $c_3>0$.
\end{itemize}
\end{theorem}
This theorem is proved in Section \ref{sec:pf thm}.
\begin{remark}
The conditions (a) and (b) together  state  that the spectral density $f_0(\lambda)$ corresponds to a short-memory  time series. The condition (b) implies that $f_0(\lambda)$ is bounded and continuous. Note that $\gamma_0(n)=O(n^{-\alpha})$ with $\alpha>1$ implies the statement (b). In view of the linear filter theory (Theorem 4.10.1 of \citet{brockwell:1991:time}), the process $\{Z_n\}$  whose spectral density is given as in (\ref{eq:f=f_df_0}) can be interpreted as a $\mathrm{FARIMA}(0,d,0)$ model with dependent noise:
\[(1-B)^d Z_n =  \xi_n,
\]
where $\{\xi_n\}$ is   short-memory  with  spectral density   $f_0(\lambda)$ and covariance function $\gamma_0(n)$,
 $B$ is the back-shift operator, and $(1-B)^d$ is interpreted using binomial series. If $f_0(\lambda)$ is chosen so that $f(\lambda)\lambda^{2d}$, $\lambda>0$,  satisfies the so-called \emph{quasi monotonicity} condition near $\lambda=0$ (see Definition 1.29 of \citet{soulier:2009:some}, which holds if $f(\lambda)\lambda^{2d}$, $\lambda>0$, is monotone in a neighborhood of $\lambda=0$), one can derive the time-domain long memory condition: the  covariance $\gamma(k)\sim c k^{2d-1}$ as $k\rightarrow\infty$ for some constant $c>0$ (see Theorem 1.37 of \citet{soulier:2009:some}).
\end{remark}

\begin{remark}\label{Rem:f_0}
The decay rate of $\gamma_0(n)$ depends on the smoothness of  $f_0(\lambda)$. The following facts are well-known:
 \begin{itemize}
 \item If   $f_0(\lambda)$ is $\alpha$-H\"older continuous with $\alpha\in (0,1]$, then
 $$
 \gamma_0(n)=O(n^{-\alpha})
 $$ (\citet{zygmund:2002:trigonometric}, Theorem 4.7),
\item and if further $\alpha>1/2$, then the absolute convergence in (b) holds (\citet{zygmund:2002:trigonometric}, Theorem 3.1);
\item if $f_0(\lambda)$ is of bounded variation, then
$$
\gamma_0(n)=O(n^{-1})
$$
 (\citet{zygmund:2002:trigonometric}, Theorem 4.12);
\item if $f_0$ is $r $ times differentiable, then
$$
\gamma_0(n)=o(n^{-r})
$$
(integration by parts and Riemann-Lebesgue lemma);
\item if $f_0(\lambda)$ has an analytic extension $f_0(z)$ to the open complex strip $\{z:\mathrm{Im(z)}<c\}$, $c>0$, then $$
    \gamma_0(n)=O(e^{-c'n})
    $$
    for any $c'\in (0,c)$ (\citet{timan:1963:theory}, Section 3.12, Formula (19)).
    \end{itemize}
\end{remark}

\begin{remark}
Let us now focus on the second term in (\ref{eq:rho_k,b bound main}). Since it involves the exponent $\alpha$ in $\gamma_0(n)$,
and since $\gamma_0(n)$ satisfies (b),
that second term may be viewed as resulting from the short-memory part.
We will only need to consider the case $\alpha\ge 1$. If $\alpha>1$ or if $\alpha=1$ and $k'=o(k)$, then the second term always tends to zero as $k\rightarrow\infty$. It resembles a  strong-mixing-type bound, which depends only on the gap $k$ separating the infinite past from the infinite future (see (\ref{eq:strong mixing})).
\end{remark}

\begin{remark}\label{Rem:diff m}
By  straightforward modifications of the proofs (including some  statements in Section \ref{sec:proof lemma}), Theorems \ref{Thm:special} and \ref{Thm:main} can be generalized for the canonical correlation $\rho\left(\mathbf{Z}_1^{a} , \mathbf{Z}_{k+1}^{k+b}\right)$ with possibly different $a$ and $b$. Under the assumptions of Theorem \ref{Thm:special}, one can obtain the bound
\[
\rho\left(\mathbf{Z}_1^{a} , \mathbf{Z}_{k+1}^{k+b}\right)\le c\big(k-\max(a,b)\big)^{2d-1}a^{1/2-d}b^{1/2-d}.
\]
Moreover, under the assumptions of Theorem \ref{Thm:main}  with $a,b <k'$, one can obtain, e.g., in the case where $\gamma_0(n)=O(n^{-\alpha})$, $\alpha>0$, the following bound:
\begin{equation}\label{eq:rho diff m}
\rho\left(\mathbf{Z}_1^{a} , \mathbf{Z}_{k+1}^{k+b}\right)\le c_1\big(k'-\max(a,b)\big)^{2d-1}a^{1/2-d}b^{1/2-d}+c_2 k'k^{-\alpha}.
\end{equation}
The bounds for the other two cases are obtained by replacing  the second term on the right-hand side of (\ref{eq:rho diff m})  as in Theorem \ref{Thm:main}.
\end{remark}

The following theorem is the key to establish the consistency of subsampling procedures for a typical class of long-memory time series models (see (\ref{eq:subordination}) below). Let $b_n$ be a sequence which will stand for the block size in the context of subsampling.
\begin{theorem}\label{Thm:block cond}
If the spectral density $f(\lambda)$ satisfies the assumptions in Theorem \ref{Thm:main} with
$$\gamma_0(n)=O(n^{-1}),
$$
 then   any
\[
b_n=o(n)
\]
implies the subsampling
condition (\ref{eq:block cond gen}), namely
$$
\sum_{k=1}^{n} \rho_{k,b_n}=o(n),    \quad \text{as }n\rightarrow\infty.
$$
\end{theorem}
This theorem is proved in Section \ref{sec:pf thm}.
The conditions in Theorem \ref{Thm:main} include some typical long-memory models:

\begin{example}[$\mathrm{FARIMA}(p,d,q)$]\label{Eg:FARIMA}
A $\mathrm{FARIMA}(p,d,q)$, $0<d<1/2$, time series $\{Z_n\}$ is the solution of the following stochastic difference equation:
\[
\Phi(B)(1-B)^d Z_n = \Theta(B) \xi_n,
\]
where $\{\xi_n\}$ is a white noise sequence with variance $\sigma^2$,   $\Phi(z)=1-\phi_1z-\ldots-\phi_pz^p$  is a polynomial with no zeros on the unit complex circle, and $\Theta(z)=1+\theta_1z+\ldots+\theta_qz^q$. See Definition 13.2.2 of \citet{brockwell:1991:time}. It is known that the spectral density of $\{Z_n\}$ is given by (Theorem 13.2.2 of \citet{brockwell:1991:time}):
\begin{equation}
f(\lambda)=\frac{\sigma^2}{2\pi} |1-e^{i\lambda}|^{-2d} \frac{|\Theta(e^{i\lambda})|^2}{|\Phi(e^{i\lambda})|^2}.
\end{equation}
The function $f_0(\lambda):=\sigma^2|\Theta(e^{i\lambda})|^2/|\Phi(e^{i\lambda})|^2$ can be extended to an analytic $f(z)$  in a strip $\{z,~\mathrm{Im}(z)<c\}$, $c>0$, since all zeros of $\Phi(z)$ are outside an open annulus containing the unit circle. By Remark \ref{Rem:f_0}, $\gamma_0(n)$ decays exponentially. Hence the assumptions of Theorem \ref{Thm:main} are satisfied and the bound (\ref{eq:rho_k,b bound main}) holds with the second term decaying exponentially. Thus, by Theorem \ref{Thm:block cond},  the block size $b_n=o(n)$ implies the subsampling condition (\ref{eq:block cond gen}).
\end{example}

\begin{example}[Fractional Gaussian noise with $H>1/2$]\label{Eg:FGN}
Fractional Gaussian noise with Hurst index $H\in (0,1)$, which is the increments of fractional Brownian motion (see, e.g., \citet{taqqu:2003:fractional}),  is a centered stationary Gaussian sequence $\{Z_{n}\}$ whose convariance is given by
\[
\gamma(n)=\frac{\sigma^2}{2}\left(|n+1|^{2H}-|n|^{2H}+|n-1|^{2H}\right), \quad \sigma^2>0.
\]
Note that $\gamma(n)\sim c n ^{2H-2}$ as $n\rightarrow\infty$ for some $c>0$.
The spectral density of $\{Z_n\}$ is derived by \citet{sinai:1976:self}, namely, for some $c_1>0$, we have
\[
f(\lambda)=c_1 |1-e^{i\lambda}|^2\sum_{k=-\infty}^{\infty}\frac{1}{|\lambda+2\pi k|^{2H+1}}.
\]
We are interested in the case where
$$
d:=H-1/2>0,
 $$
 the long-memory regime. We mention that in the so-called antipersistent regime $H<1/2$, one can show using (\ref{eq:crude bound rho_k,b}) that the condition (\ref{eq:block cond gen})  holds under $b_n=o(n^{1-\epsilon})$ for an $\epsilon>0$ arbitrarily small. See \citet{bai:taqqu:zhang:2015:unified}.  Write the spectral density as
\[
f(\lambda)=f_d(\lambda) f_0(\lambda),
\]
where $f_d$ is as in (\ref{eq:f_d}), and
\begin{align}\label{eq:f_0 fgn}
f_0(\lambda)&=c_1 |1-e^{i\lambda}|^{2d+2} \sum_{k=-\infty}^\infty  \frac{1}{|\lambda+2\pi k|^{2d+2}}\notag\\
&=c_1\left[ \left(\frac{4\sin(\lambda/2)^2}{\lambda^2}\right)^{d+1} + \left(4\sin(\lambda/2)^2\right)^{d+1}  \sum_{k\neq 0}  \frac{1}{|\lambda+2\pi k|^{2d+2}}
\right].
\end{align}
In view of the term $\left(4{\sin(\lambda/2)^2}/{\lambda^2}\right)^{d+1}$, we see that $f_0(\lambda)$ is  bounded below away from $0$ when $\lambda\in (-\pi,\pi]$. Furthermore,  note that since $d\in (0,1/2)$, the series $\sum_{k\neq 0}  |\lambda+2\pi k|^{-2d-2}$ and its term-by-term derivatives (with respect to $\lambda$)  converge uniformly for all $\lambda\in(-\pi,\pi]$.  By  Theorem 7.17 of \citet{rudin:1976:principles}, the series and hence
$f_0(\lambda)$ is infinitely differentiable. Then the assumptions of Theorem \ref{Thm:main} hold with arbitrarily large $\alpha>0$ in Condition (3) (see Remark \ref{Rem:f_0}), and hence so does the bound (\ref{eq:rho_k,b bound main}). Thus, by Theorem \ref{Thm:block cond},  the block size $b_n=o(n)$ implies the subsampling condition (\ref{eq:block cond gen}).
\end{example}

\subsection{Multivariate extension}\label{Sec:multi var}

Theorem   \ref{Thm:block cond} can be directly extended to the case of a multivariate second-order stationary process
 $$
\{\mathbf{Z}_n=(Z_{n,1},\ldots,Z_{n,J})\}
 $$
 with \emph{uncorrelated} components. One can define similarly the canonical correlation $\rho_{k,b}$ between  two blocks $(\mathbf{Z}_{1},\ldots \mathbf{Z}_b)$ and $(\mathbf{Z}_{k+1},\ldots,\mathbf{Z}_{k+b})$. Let $\rho_{k,b,j}$ be the   canonical correlation between $(Z_{1,j},\ldots,Z_{b,j})$ and $(Z_{k+1,j},\ldots,Z_{k+b,j})$ involving the $j$-th component of the vectors, $j=1,\ldots,J$.   Since the covariance between $Z_{n_1,j_1}$ and $Z_{n_2,j_2}$ vanishes if $j_1\neq j_2$, the  covariance matrices $\Sigma_b$ in (\ref{eq:Sigma_m}) and $\Sigma_{k,b}$ in (\ref{eq:Sigma_k,b})  corresponding to $\{\mathbf{Z}_n\}$ are block-diagonal, and hence so are the $U_{k,b}$ and $V_{k,b}$ matrices defined in (\ref{eq:U matrix}) and (\ref{eq:V matrix}) below. Because the eigenvalues of a block-diagonal matrix consists of eigenvalues of each block, one has in view of Lemma \ref{Lem:can corr matrix} below that
\begin{equation}\label{eq:rho_k,b max}
\rho_{k,b}=
\max_{1\le j\le J}\rho_{k,b,j}.
\end{equation}
Hence to establish the subsampling condition (\ref{eq:block cond gen}) for $\rho_{k,b}$ of $\{\mathbf{Z}_n\}$ which has uncorrelated components, one only needs  to establish the corresponding ones for  each component $\{Z_{n,j}\}$. See Example \ref{eg:nonlinear} below for an application of such a multivariate case to a nonlinear time series.

\section{Validity of subsampling}\label{sec:app subsampling}
In this section, we discuss the role that the subsampling condition (\ref{eq:block condition alpha}) plays in ensuring the consistency of a subsampling procedure.

\subsection{Subsampling of time series}
We start with a brief introduction to  subsampling procedures for time series.  Let $\{X_n\}$ be a stationary time series. One is interested in using the quantity
$$
T_n(\mathbf{X}_1^n;\theta)=T_n(X_1,\ldots,X_n;\theta)
 $$
 for inference, where $\theta$ is an unknown parameter, which may not be present in some cases.  Suppose that
\begin{equation}\label{eq:T_n ConvD}
T_n(\mathbf{X}_1^n;\theta)\ConvD T
\end{equation}
as $n\rightarrow\infty$, where $T$ is some random variable with distribution function $F_T(x)$. In general, the limit $T$ depends on $\theta$. But we suppress this in the notation for simplicity. Whenever we mention convergence of distribution functions, we mean convergence at continuity points of the limit.  The  convergence (\ref{eq:T_n ConvD}) is established on a case-by-case basis.

 We are interested in the distribution function
\begin{equation} \label{eq:F_T_n}
{F}_{T_n}(x)=P(T_n(\mathbf{X}_1^n;\theta)\le x),
\end{equation}
which is in general  difficult to obtain. Suppose that the limit distribution $F_T(x)$ is  not available either, due for example,  to the presence of nuisance parameters. So we resort to some resampling procedure. Consider  the statistic  defined on a  block of length  $b$ starting at $i$, namely,
$$
T_b(\mathbf{X}_i^{i+b-1};\theta).
 $$
 One expects that the distribution of $T_b(\mathbf{X}_i^{i+b-1};\theta)$ is close to that of $T_n(\mathbf{X}_1^n;\theta)$   since, in view of (\ref{eq:T_n ConvD}), both are close to $F_T$ when $b$ and $n$ are reasonably large. By varying $i$ while keeping $b$ fixed, one obtains many subsamples. One then wants to use   the empirical distribution
\begin{equation}\label{eq:emp distr}
\widehat{F}_{n,b}^*(x)=\frac{1}{n-b+1}\sum_{i=1}^{n-b+1} \mathrm{I}\{T_b(\mathbf{X}_i^{i+b-1};\theta)\le x\}
\end{equation}
as an approximation of $F_{T_n}$ for inference. But $\widehat{F}_{n,b}^*(x)$ involves the unknown parameter $\theta$. We thus replace $\theta$ by a consistent estimate $\widehat{\theta}_n$ which depends on the whole sample $\{X_1,\ldots,X_n\}$. This leads to
\begin{equation}\label{eq:emp distr no star}
\widehat{F}_{n,b}(x)=\frac{1}{n-b+1}\sum_{i=1}^{n-b+1} \mathrm{I}\{T_b(\mathbf{X}_i^{i+b-1};\widehat{\theta}_n)\le x\}.
\end{equation}
Note that each term $\widehat{F}_{n,b_n}(x)$ depends on the whole sample $\{X_1,\ldots,X_n\}$ while each term of $\widehat{F}_{n,b}^*(x)$  depends only on the an individual block  $\{X_i,\ldots,X_{i+b-1}\}$.
On the other hand $\widehat{F}_{n,b_n}(x)$ is computable from data since it does not involve the unknown parameter $\theta$.

\begin{example}
As  a typical case, suppose  (see Chapter 3 of \citet{politis:1999:subsampling}) that $\theta$ is indeed the parameter on which we want to carry out the inference. Let $\widehat{\theta}_{n,b,i}$ be an  estimator of $\theta$ computed using the block  $\{X_i,\ldots,X_{i+b-1}\}$. To get an approximation of the distribution of $\widehat{\theta}_n-\theta$ for inference, one then proposes to use the empirical distribution
\begin{equation}\label{eq:emp distr theta}
\widehat{F}_{n,b}(x)=\frac{1}{n-b+1}\sum_{i=1}^{n-b+1} \mathrm{I}\{\tau_b(\widehat{\theta}_{n,b,i}-\widehat{\theta}_n) \le x\},
\end{equation}
where  $\tau_b$ is an appropriate deterministic normalizer which ensures that $\tau_n(\widehat{\theta}_n-\theta)$ converges in distribution as $n\rightarrow\infty$.   In this case, we have
\[
T_b(\mathbf{X}_i^{i+b-1};\theta)=\tau_b(\widehat{\theta}_{n,b,i}-\theta),\qquad
T_b(\mathbf{X}_i^{i+b-1};\widehat{\theta}_n)=\tau_b(\widehat{\theta}_{n,b,i}-\widehat{\theta}_n).
\]
If the convergence rate $\tau_n$ is unknown, it needs to be consistently estimated by some $\widehat{\tau}_n$  in the sense that $\widehat{\tau}_n/\tau_n$ converges in probability to a positive constant (see \citet{politis:1999:subsampling}, Chapter 8). In the cases of heavy tails or long memory, often  $\tau_b$  needs to be  replaced by a random  normalizer $\widehat{\tau}_{n,b,i}$ computed using the block $\{X_i,\ldots,X_{i+b-1}\}$, and $\widehat{\tau}_{n,b,i}/\tau_b$ would typically only converge in distribution as $b\rightarrow\infty$. This is done, e.g., in \citet{romano:wolf:1999:subsampling}, \citet{jach:2012:subsampling} and \citet{bai:taqqu:zhang:2015:unified}. We mention that even without a proper scale estimate $\widehat{\tau}_n$, the shape of the unscaled empirical distribution
$$
\frac{1}{n-b+1}\sum_{i=1}^{n-b+1} \mathrm{I}\{(\widehat{\theta}_{n,b,i}-\widehat{\theta}_n) \le x\}
$$
can be useful for diagnostic purposes (see \citet{sherman:carlstein:1996:replicate}).
\end{example}
\begin{example}
There are cases where no  unknown  $\theta$ is involved. For example in change-point problems using the Wilcoxon test statistics, which involves $\sum_{i=1}^k R_i -k/n \sum_{i=1}^n R_i$ where $R_i$'s are ranks. In this case, one can suppress the  $\theta$ in (\ref{eq:emp distr}) and the $\widehat{\theta}$ in (\ref{eq:emp distr no star}), and hence the distinction between $\widehat{F}_{n,b}^*(x)$ and $\widehat{F}_{n,b}(x)$. See \citet{betken:Wendler:2015:subsampling}.
\end{example}

\subsection{Asymptotic consistency of subsampling procedures}\label{sec:asymp consis}
To obtain convergence results, we let the block size $b=b_n$ depend explicitly on the sample size $n$,  which  tends to infinity as $n\rightarrow\infty$.
\begin{definition}\label{Def:consistency}
Let $\widehat{F}_{n,b_n}(x)$ be as  in (\ref{eq:emp distr no star}) and let $F_{T_n}(x)$ be as in (\ref{eq:F_T_n}).
We say that the subsampling procedure is consistent, if
\begin{equation}\label{eq:subsample consistency}
|\widehat{F}_{n,b_n}(x)-F_{T_n}(x)|\ConvP 0
\end{equation}
as $n\rightarrow \infty$ for $x$ at the continuity point of the  limit distribution $T$.
\end{definition}
By some standard argument using Polya's Theorem (see the proofs of Theorem 3.1 and Corollary 3.1 in \citet{bai:taqqu:zhang:2015:unified}), if $F_T(x)$ is continuous, one can have a stronger form of consistency, namely,
$$
\sup_x|\widehat{F}_{n,b_n}(x)-F_{T_n}(x)|\ConvP 0.
$$

When proving the consistency of the empirical distribution of subsampling, for example of
$\widehat{F}_{n,b_n}$   in (\ref{eq:emp distr theta}),  a common strategy is  to first replace $T_b(\mathbf{X}_i^{i+b-1};\widehat{\theta}_n)$, if necessary,  by $T_b(\mathbf{X}_i^{i+b-1};\theta)$, so that it depends only on the block $\{X_i,\ldots,X_{i+b-1}\}$. One needs to show that this modification is asymptotically  negligible (see \citet{politis:1999:subsampling}, the proofs of Theorem 3.2.1 and Theorem 11.3.1).
After this reduction, we are basically working with $\widehat{F}^*_{n,b_n}(x)$ in (\ref{eq:emp distr}). To establish the consistency (\ref{eq:subsample consistency}),  we then only need to show  that
\begin{equation}\label{eq:F_b,n F_T}
 \widehat{F}_{n,b_n}^*(x) \ConvP F_T(x) ,
\end{equation}
since $F_{T_n}(x)\rightarrow F_T(x)$ by the assumption (\ref{eq:T_n ConvD}).
 To do so, we write the
bias-variance decomposition of the mean-square error:
\begin{align}\label{eq:bias-var}
\E \left| \widehat{F}_{n,b_n}^*(x)-F_T(x) \right|^2&=\left[\E\widehat{F}_{n,b_n}^*(x) -F_T(x)\right]^2+ \left[\E\left(\widehat{F}_{n,b_n}^*(x)^2\right)-\left(\E \widehat{F}_{n,b_n}^*(x)\right)^2\right]\notag\\&=\left|F_{T_{b_n}}^*(x)- F_T(x)  \right|^2+\Var\left[\widehat{F}_{n,b_n}^*(x)\right]
\end{align}
since
$$
\E\widehat{F}_{n,b_n}^*(x) = F_{T_{b_n}}^*(x)
$$
 in view of  (\ref{eq:emp distr}) and (\ref{eq:F_T_n}).
The first term in (\ref{eq:bias-var}) is the
 bias term, which converges to zero due to the assumption $b_n\rightarrow \infty$ and (\ref{eq:T_n ConvD}). The key is to bound the second variance term. By a standard argument (see the proof of Theorem 3.1 of \citet{bai:taqqu:zhang:2015:unified}),
\begin{equation}\label{eq:variance bound}
\Var\left[\widehat{F}_{n,b_n}^*(x)\right]\le \frac{2}{n-b_n+1}\sum_{k=0}^{n-b_n+1} \left|\Cov[\mathrm{I}\{ T_{b_n}(\mathbf{X}_1^{b_n};\theta)\le x \}, \mathrm{I}\{T_{b_n}(\mathbf{X}_{k+1}^{k+b_n};\theta)\le x\}  ]\right|.
\end{equation}
Note that $\mathrm{I}\{ T_{b_n}(\mathbf{X}_{k+1}^{k+b_n};\theta)\le x \}$ is a function of $\{X_{k+1},\ldots,X_{k+b_n}\}$ which is measurable with respect to the sigma field $\mathcal{F}_{k+1}^{k+b_n}$. Because $\Cov[\mathrm{I}_A,\mathrm{I}_B]=P(A)P(B)-P(A\cap B)$, we have
\[
\left|\Cov[\mathrm{I}\{ T_{b_n}(\mathbf{X}_1^{b_n};\theta)\le x \}, \mathrm{I}\{T_{b_n}(\mathbf{X}_{k+1}^{k+b_n};\theta)\le x\}  ]\right|\le \alpha_{k,b_n},
\]
where $\alpha_{k,b_n}$ is the between-block mixing coefficient defined in (\ref{eq:alpha_{k,b}}).
Hence from the variance bound (\ref{eq:variance bound}), one has by the subsampling condition (\ref{eq:block condition alpha}) that
\begin{equation}\label{eq:variance bound alpha}
\Var\left[\widehat{F}_{n,b_n}^*(x)\right]\le \frac{2}{n-b_n+1}\sum_{k=0}^{n-b_n+1} \alpha_{k,b_n} \rightarrow 0
\end{equation}
when $n\rightarrow\infty$ and $b_n=o(n)$.
For convenience, we formulate below a corresponding result in  the context of Gaussian subordination  (\ref{eq:subordination}).
\begin{proposition}\label{Pro:var}
Suppose that $\{X_n\}$ follows the Gaussian subordination model (\ref{eq:subordination}) and let $\widehat{F}_{n,b}^*(x)$ be defined as in (\ref{eq:emp distr}). If the subsampling condition (\ref{eq:block cond gen}) holds and $b_n=o(n)$, we have
\begin{equation*}
\lim_{n\rightarrow\infty} \Var\left[\widehat{F}_{n,b_n}^*(x)\right]= 0.
\end{equation*}
\end{proposition}
\begin{proof}
In view of (\ref{eq:alpha<=rho}), the condition (\ref{eq:block cond gen}) implies the condition (\ref{eq:block condition alpha}).  Then apply (\ref{eq:variance bound alpha}).
\end{proof}
This proposition shows why the subsampling condition (\ref{eq:block cond gen}) is useful. Consequently, we have
\begin{theorem}\label{Thm:subsample consistency}
Assume that the following hold:
\begin{enumerate}
 \item $\{X_n\}$ is given by the Gaussian subordination model (\ref{eq:subordination}), where the underlying Gaussian  $\{Z_n\}$ satisfies the assumptions of Theorem \ref{Thm:block cond}.
\item
 As $n\rightarrow\infty$,  the convergence (\ref{eq:T_n ConvD}) holds.
\item Let $x$ be any continuity point of $F_T(x)=P(T\le x)$. Assume that for any $\epsilon>0$ and $\delta>0$, with probability tending to $1$ as $n\rightarrow\infty$, one has the following:
\[
\widehat{F}_{n,b_n}^*(x-\epsilon)-\delta\le \widehat{F}_{n,b_n}(x)\le  \widehat{F}^*_{n,b_n}(x+\epsilon)+\delta.
\]
\end{enumerate}
Then if $b_n\rightarrow \infty$ and $b_n=o(n)$ as $n\rightarrow\infty$, the consistency of subsampling in the sense of Definition \ref{Def:consistency} holds, namely,
\[|F_{T_n}(x)-\widehat{F}_{n,b_n}(x)|\ConvP 0\]
for $x$ at any continuity point of $F_T$.

\begin{proof}
Assumption 1 yields Proposition \ref{Pro:var}. By Assumption 2, we have $F_{T_n}(x)\rightarrow F_T(x)$ as $n\rightarrow\infty$. Combining Proposition \ref{Pro:var} and Assumption 2, we get
 $$
 \widehat{F}^*_{n,b_n}(x) \ConvP F_T(x)
  $$
  by (\ref{eq:bias-var}).  Combining this further with  Assumption 3, and using the arguments in the proof of  Theorem 3.1 of \citet{bai:taqqu:zhang:2015:unified} yields
  $$\widehat{F}_{n,b_n}(x)\ConvP F_T(x).
  $$
    The conclusion  then follows from the triangle inequality
 \[
 |F_{T_n}(x)-\widehat{F}_{n,b_n}(x)|\le |F_{T_n}(x)-F_T(x)|+|\widehat{F}_{n,b_n}(x)- F_T(x)|.
\]
\end{proof}
\begin{remark}
Assumption 3 needs to be checked in a case-by-case basis. See, e.g.,  the proofs of  Theorem 3.1 of \citet{bai:taqqu:zhang:2015:unified} and  Theorem 11.3.1 of \citet{politis:1999:subsampling}. It may   take different forms depending on the specific problem at hand.
\end{remark}

\begin{remark}
Theoretical understanding of the optimal choice of block size $b$ in subsampling has been in general a difficult problem (see Chapter 9 of \citet{politis:1999:subsampling}). Our results have some heuristic implication for the choice of $b$. Consider the bias-variance decomposition in (\ref{eq:bias-var}) for the empirical distribution $\widehat{F}_{n,b_n}^*(x)$ (which involves the unknown $\theta$) instead of that of $\widehat{F}_{n,b_n}(x)$. For simplicity, let  $\{X_n=Z_n\}$ be a $\mathrm{FARIMA}(0,d,0)$ Gaussian process.  From (\ref{eq:alpha<=rho}),  (\ref{eq:variance bound alpha}) and Theorem \ref{Thm:special}, one can derive    that when $b_n\ll n$, the variance term satisfies
\begin{align}\label{eq:var part}
\Var\left[\widehat{F}_{n,b}^*(x)\right]\le& \frac{2}{n-b_n+1}\sum_{k=0}^{n-b_n+1} \rho_{k,b_n}\le \frac{c}{n}
\left[\sum_{k=0}^{b_n}\rho_{k,b_n} +\sum_{k=b_n+1}^n \rho_{k,b_n}\right]\notag\\
\le &c_1 \frac{b_n}{n} +  \frac{c_2}{n}\left(\sum_{k=1}^n k^{2d-1}  \right) b_n^{1-2d}   \le c_1 \frac{b_n}{n}
+ c_2 \left(\frac{b_n}{n}\right)^{1-2d}\le c_3\left(\frac{b_n}{n}\right)^{1-2d}.
\end{align}
On the other hand, one also needs the rate at which the bias term $[F_{T_{b}}^*(x)- F_T(x) ]^2$
in (\ref{eq:bias-var})
tends to zero.  It is in  general difficult to assess such a rate, unless $T$ is some special random variable, for example a  Gaussian. Suppose such a rate is available: $[F_{T_{b}}^*(x)- F_T(x) ]^2\asymp b^{-2\gamma}$ for some  $\gamma>0$. Suppose also that the behavior of the variance is captured exactly by the bound in (\ref{eq:var part}).  Then by  solving $\left({b/n}\right)^{1-2d}\asymp b^{-2\gamma}$, one gets the optimal order
$$
b\asymp n^{(1-2d)/(1-2d+2\gamma)}.
 $$
 Under weak dependence, the usual Berry-Esseen rate of Gaussian approximation is $\gamma=1/2$.  If one supposes that the rate is given by $\gamma=1/2-d$ under long memory, then an optimal order of $b$ would be
\[b\asymp n^{1/2},\]
which has been empirically observed to perform well in the context of subsampling under long memory (see e.g., \citet{hall:1998:sampling},  \citet{betken:Wendler:2015:subsampling} and \citet{zhang:2013:block}). Note that the preceding argument is only heuristic and is not necessarily correct. It does not apply for example to the case of the subsampling estimation of the scale parameter $\sigma_{n,d}:=n^{-1-2d}\Var[\sum_{i=1}^n Z_i]$ considered in \citet{kim:nordman:2011:properties}, where the optimal order can be $b_n\asymp n^{a}$ for $a$ ranging to the whole interval $(0,1)$, depending on $d$.
\end{remark}

\end{theorem}

\subsection{Application to cases considered in the literature}
Based on Theorem \ref{Thm:block cond}, some results on  subsampling procedures under the Gaussian subordination model (\ref{eq:subordination}) considered in the literature can be established or improved as indicated below.

\begin{example}\label{eg:hall}
\citet{hall:1998:sampling} considered subsampling for the sample mean. Their assumptions, however,  allow only the case where the underlying Gaussian $\{Z_n\}$ in (\ref{eq:subordination}) is completely regular, which is equivalent to strong mixing (see \citet{ibragimov:rozanov:1978:gaussian}, Section IV.1), but not the long memory regime considered here.
Replacing their assumptions on $\{Z_n\}$ by the ones described in Theorem \ref{Thm:block cond}, the condition $b_n=o(n)$ will yield consistency of their subsampling procedure.  See also \citet{lahiri:2003:resampling}, Section 10.4.
\end{example}

\begin{example}
\citet{conti:2008:confidence} considered subsampling for an estimator of the long memory parameter. They assumed as in \citet{hall:1998:sampling} that  $\{Z_n\}$ is completely regular.
Replacing their assumptions on $\{Z_n\}$ by the ones given in Theorem \ref{Thm:block cond}, the condition $b_n=o(n)$ will yield consistency of their subsampling procedure.
\end{example}

\begin{example}
\citet{psaradakis:2010:inference} studied subsampling for the one-sample sign statistic under the same framework of \citet{hall:1998:sampling}. Replacing the assumptions on $\{Z_n\}$  in Theorem \ref{Thm:block cond}, the block size condition $b_n=o(n)$ will yield consistency of the subsampling procedure.
\end{example}

 \begin{example}
\citet{betken:Wendler:2015:subsampling} considered subsampling for a general statistic as in (\ref{eq:emp distr}) under the model (\ref{eq:subordination}), and also discussed a robust change-point test as an example. If their assumptions on $\{Z_n\}$ are replaced by those of Theorem 2.3 (a), their block size condition $b_n=o(n^{1-d-\epsilon})$ with arbitrarily small $\epsilon>0$    can  be relaxed to $b_n=o(n)$ and  consistency still holds for a general statistic.
 \end{example}

\begin{example}\label{eg:bai taqqu zhang}
 \citet{bai:taqqu:zhang:2015:unified} studied subsampling  for sample mean with the self-normalization considered in \citet{shao:2010:self}, which avoids dealing with various nuisance parameters. They adopted the assumptions in Theorem \ref{Thm:block cond}, so that $b_n=o(n)$ yields  consistency. See \citet{bai:taqqu:zhang:2015:unified}, Corollary 3.1.
 \end{example}

\begin{example}\label{eg:nonlinear}
Suppose that $\{X_n\}$ is subordinated to a multivariate stationary Gaussian process
$$
\{\mathbf{Z}_n=(Z_{n,1},\ldots,Z_{n,J})\}
 $$
 with independent components.  As noted in Section \ref{Sec:multi var}, the analog of $\rho_{k,b}$  in (\ref{eq:rho_k,b}) for $\{\mathbf{Z}_n\}$ is  the maximum of the $\rho_{k,b,j}$'s computed from the individual  $\{Z_{n,j}\}$'s.  This fact is useful when one  considers  some nonlinear time series models.

For example, let
$$
\mathbf{Z}_n=(\eta_n,\xi_n)
 $$
 be a bivariate stationary Gaussian process, where   $\{\eta_n\}$ satisfies the long memory conditions as in Theorem \ref{Thm:block cond}, and $\{\xi_n\}$ is i.i.d.\ standard normal and is independent of $\{\eta_n\}$. Consider the following stochastic volatility model (see, e.g.,  \citet{beran:2013:long}, Section 2.1.3.8):
\begin{equation}\label{eq:stoch vol}
X_n=G(\eta_n)F(\xi_n),
\end{equation}
where $Z_{n,1}=\eta_n$ and $Z_{n,2}=\xi_n$,
and $G(\cdot)$ is a positive function and $F(\cdot)$ is chosen so that $\E F(\xi_n)=0$. The flexibility of choosing $G(\cdot)$ and $F(\cdot)$ allows a variety of marginal distributions for $\{X_n\}$. The sequence $\{X_n\}$ is uncorrelated, but the volatility $G(\eta_n)$ of $\{X_n\}$ may display long memory if $G(\cdot)$ is chosen appropriately.  See also  \citet{deo:hsieh:hurvich:2010:long} and \citet{jach:2012:subsampling} for variants of (\ref{eq:stoch vol}).

Note that in (\ref{eq:rho_k,b max}), the  canonical correlation $\rho_{k,b,1}$ for $\{\eta_n\}$  dominates since the canonical correlation   $\rho_{k,b,2}$ for $\{\xi_n\}$ vanishes when $k>b$ due to independence. Hence if $b_n=o(n)$, the  subsampling condition (\ref{eq:block cond gen}) holds for the underlying $\{\mathbf{Z}_n\}$ because it holds for $\{\eta_n\}$ by Theorem \ref{Thm:block cond}.
\end{example}

\begin{remark}\label{rem:other model}
In the special case of subsampling for the sample mean with a {\it deterministic} normalization,   where $\{X_n\}$ is a long-memory linear process which is not necessarily Gaussian, \citet{nordman:2005:validity}  showed that a block size condition $b_n=o(n^{1-\epsilon})$ for any $\epsilon>0$ suffices for consistency. In the same setup but with  $\{X_n\}$ replaced by a nonlinear function of a long-memory linear process, \citet{zhang:2013:block} obtained consistency with the same block size condition $b_n=o(n^{1-\epsilon})$ for arbitrarily small $\epsilon>0$. These do not fall within the Gaussian subordination framework (\ref{eq:subordination}). Establishing (\ref{eq:block condition alpha}) for long memory models beyond Gaussian subordination under a non-restrictive condition close to $b_n=o(n)$ may be an interesting open problem.
\end{remark}

\section{Subsampling of common statistics}\label{sec:other common stat}

In the following sections, we discuss subsampling of some  common types of statistics under long memory. The case of the sample mean is discussed in Examples \ref{eg:hall},  \ref{eg:bai taqqu zhang} and Remark \ref{rem:other model}.
 See also Chapter 4 of \citet{lahiri:2003:resampling} for  expositions about block bootstrap methods under weak dependence.
\subsection{Sample autocovariance}\label{sec:cov}
Let $\{Z_n\}$ be a stationary long-memory Gaussian process satisfying the assumptions in Theorem \ref{Thm:block cond}  and  $\gamma(k)=\Cov[Z_k,Z_0]\sim c  k^{2d-1}$ for some $c >0$  as $k\rightarrow\infty$. Let the memory parameter be $d\in (0,1/2)$.
 For simplicity, we focus on the  case where
 $$
 X_n=Z_n.
 $$ We consider  the estimation of $\gamma(m)$ by the sample autocovariance
\[
\widehat{\gamma}(m)=\frac{1}{n}\sum_{i=1}^{n-b} (X_i-\bar{X}_n)(X_{i+m}-\bar{X}_n),
\]
where $\bar{X}_n=n^{-1}(X_1+\ldots+X_n)$. We treat for simplicity here only the subsampling inference for the $\gamma(m)$ at a single lag $m$, while the joint inference at different lags can also be considered in the same way.   Even under the assumption of $\{X_n\}$ being Gaussian, the asymptotic behavior of $\widehat{\gamma}_n(m)$ is intricate.
Indeed, we have (see, e.g., \citet{dehling:taqqu:1991:bivariate} and \citet{hosking:1996:asymptotic}) as $n\rightarrow\infty$ that
\begin{equation}\label{eq:sample cov diff}
\tau_n\big( \widehat{\gamma}_n(m)-\gamma(m) \big)\ConvD\begin{cases}
 N(0,\sigma_1^2),& d\in (0,1/4), ~\tau_n= n^{1/2}; \\
\sigma_2 R_d,  & d\in (1/4,1/2),
~\tau_n=n^{1-2d},
\end{cases}
\end{equation}
where $\sigma_1$ and $\sigma_2$ are positive scale constants, and $R_d$ is a non-Gaussian distribution termed ``modified Rosenblatt distribution'' by \citet{hosking:1996:asymptotic}.
Using (\ref{eq:sample cov diff}) directly for asymptotic inference involves some difficulties: 1) the dichotomy of asymptotic  distributions as the memory parameter $d$ varies, 2)  $\tau_n$  takes a non-standard  rate if $d>1/4$ which is unknown. This forces one to deal with several nuisance parameters.

Now consider using subsampling for inference. We need to account for the non-standard scaling $\tau
_n$. Let $\widehat{\tau}_n=\widehat{\tau}_n(\mathbf{X}_1^n)$ be a random normalizer computed from the data so that
\begin{equation}\label{eq:cov conv normalize}
T_n:=\widehat{\tau}_n\big(\widehat{\gamma}_n(m)-\gamma(m)\big)\ConvD T
\end{equation}
as $n\rightarrow\infty$
for some non-degenerate random variable $T$.
 It is practically attractive to have just one   unified form of the normalizer  $\widehat{\tau}_n$ which works across  the different cases in (\ref{eq:sample cov diff}).  Such  $\widehat{\tau}_n$ can be in principle achieved by the ``self-normalization'' considered in \citet{lobato:2001:testing}, \citet{shao:2010:self}, although some functional limit theorems involved are yet to be established. See also \citet{McElroy:Jach:2012:Subsampling} for a different self-normalization.  We require that the  self-normalization $\widehat{\tau}_n$ captures the scaling $\tau_n$ in the sense of that as $n\rightarrow\infty$,
\begin{equation}\label{eq:self-norm ass}
\tau_n/\widehat{\tau}_n \ConvD W
\end{equation}
for some  random variable $W$ satisfying $P(W\neq 0)=1$. Now the conditions (\ref{eq:sample cov diff}), (\ref{eq:cov conv normalize}) and  (\ref{eq:self-norm ass}) are in line  with Assumption 11.3.1 of \citet{politis:1999:subsampling}.

Define now  analogously the sample autocovariance  computed on the block $\mathbf{X}_{k+1}^{k+b}$:
\[
\widehat{\gamma}_b(m;\mathbf{X}_{k+1}^{k+b})=\frac{1}{b}\sum_{i=k+1}^{k+b-m} (X_i-\bar{X}_{k+1}^{k+b})(X_{i+m}-\bar{X}_{k+1}^{k+b}),
\]
where
$$
\bar{X}_{k+1}^{k+b}=b^{-1}(X_{k+1}+\ldots+X_{k+b}),\ k=0,1,\ldots,n-b,\ 0\le m\le b.
$$
Set the statistics on the block as
\begin{equation}\label{eq:cov block stat}
T_b(\mathbf{X}_{k+1}^{k+b};\widehat{\gamma}(m))=\widehat{\tau}_b(\mathbf{X}_{k+1}^{k+b})\Big( \widehat{\gamma}_b(m;\mathbf{X}_{k+1}^{k+b})-\widehat{\gamma}_n(m) \Big),
\end{equation}
where $\widehat{\gamma}_b(m;\mathbf{X}_{k+1}^{k+b})$ satisfies (\ref{eq:self-norm ass}) and is based on the block $\mathbf{X}_{k+1}^{k+b}$.

\medskip
We formulate a result as follows, which is within the framework of Theorem \ref{Thm:subsample consistency}. The proof is similar to  that of Theorem 11.3.1 of \citet{politis:1999:subsampling}. See also the proof of Theorem 3.1 of \citet{bai:taqqu:zhang:2015:unified}.
\begin{theorem}\label{Thm:cov}
Let $\{X_n=Z_n\}$ be a long-memory Gaussian process satisfying the assumptions in Theorem \ref{Thm:block cond}. Assume that (\ref{eq:sample cov diff}), (\ref{eq:cov conv normalize}) and (\ref{eq:self-norm ass}) hold. Let \[\widehat{F}_{n,b}(x)=(n-b+1)^{-1}\sum_{i=1}^{n-b+1} \mathrm{I}\{ T_b\big(\mathbf{X}_{i}^{i+b-1};\widehat{\gamma}(m)\big)\le x\},\] where $T_b(\mathbf{X}_{i}^{i+b-1};\widehat{\gamma}_n(m))$ is as in (\ref{eq:cov block stat}). Let $F_{T_n}(x)$ be the distribution function of  $T_n$ in (\ref{eq:cov conv normalize}). Then  the consistency of the subsampling procedure holds: if $n\rightarrow\infty$, $b_n\rightarrow\infty$ and $b_n=o(n)$, we have
\[
|F_{T_n}(x)-\widehat{F}_{n,b_n}(x)|\ConvP 0
\]
for $x$ at the continuity point of the distribution of $T$ in (\ref{eq:cov conv normalize}).
\end{theorem}

Sample autocovariance falls within the category called ``smooth function of mean''  (see, e.g., Example 4.4.2 of \citet{politis:1999:subsampling} and Section 4.2 of \citet{lahiri:2003:resampling}).   Theorem \ref{Thm:cov} may be extended to this general category, given that asymptotic results analogous to (\ref{eq:sample cov diff}), (\ref{eq:cov conv normalize}) and  (\ref{eq:self-norm ass}) are established specifically.
\subsection{M-estimation}
\citet{beran:1991:m} considered the M-estimation for the following location model
\begin{equation}\label{eq:location model}
X_i=\mu+Q(Z_i),
\end{equation}
where $\{Z_i \}$ is a \emph{standardized} long-memory Gaussian process satisfying the assumptions in Theorem \ref{Thm:block cond}, and $\gamma(k)=\Cov[Z_k,Z_0]\sim c  k^{2d-1}$ for some $c >0$  as $k\rightarrow\infty$.  The  function $Q(\cdot)$ satisfies $\E Q(Z_i)=0$ and $\sigma^2:=\E Q(Z_i)^2<\infty$. Assume for simplicity that $\sigma^2=1$, while in general $\sigma^2$ can be consistently estimated by the sample variance and this does not affect the asymptotic results.
The estimating equation  is given by
\begin{equation}\label{eq:est eq}
\sum_{i=1}^n \psi(X_i-x)=0,
\end{equation}
where $\psi$ is some deterministic function such that $\E \psi (X_i-x)=0 $ if and only if $x=\mu$.
Let $\widehat{\mu}_n$ be the resulting M-estimator of $\mu$ (the solution to (\ref{eq:est eq})). Theorem 1 of \citet{beran:1991:m} states that
\begin{equation}\label{eq:m est normal}
n^{1/2-d}(\widehat{\mu}_n-\mu)\ConvD N(0,\sigma_M^2)
\end{equation}
for some scale constant $\sigma_M>0$. Theorem 1 of \citet{beran:1991:m} has five assumptions: the assumptions 1-4 are standard regularity conditions in the M-estimation context, while the 5th one is imposed to restrict to the Gaussian  asymptotics   in
(\ref{eq:m est normal}). If the 5th assumption  is dropped, depending on the Hermite rank $m$ (see (\ref{eq:Hermite Rank Def}))  of the composite transform $\psi\circ Q$, one may have
\begin{equation}\label{eq:m est nonnormal}
\tau_n(\widehat{\mu}_n-\mu)\ConvD  \begin{cases}
 N(0,\sigma_1^2),& (2d-1)m<-1,~ \tau_n= n^{1/2}; \\
\sigma_2 Z_{m,d},  & (2d-1)m>-1,  ~
\tau_n=n^{(d-1/2)m+1},
\end{cases}
\end{equation}
where $\sigma_1$ and $\sigma_2$ are positive scale constants, and $Z_{m,d}$ is the so-called \emph{Hermite distribution} (see \citet{dobrushin:major:1979:non}) which is non-Gaussian if the Hermite rank $m\ge 2$. Using directly (\ref{eq:m est nonnormal}) for inference is again difficult due to the dichotomy and the nuisance parameters.

However, as in Section \ref{sec:cov}, we can consider  a subsampling procedure with a proper self-normalization $\widehat{\tau}_n$ (see, e.g., \citet{shao:2010:self}). We omit the formal statement of the result, which is   similar to Theorem \ref{Thm:cov}. Note that  asymptotic results similar to (\ref{eq:sample cov diff}), (\ref{eq:cov conv normalize}) and  (\ref{eq:self-norm ass}) need to be established.

\subsection{Empirical process}
Let $\{Z_n\}$  be long-memory Gaussian with memory parameter $d\in (0,1/2)$ that satisfies the assumptions in Theorem \ref{Thm:block cond} as well as $\Cov[Z_k,Z_0]\sim c k^{2d-1}$ for some constant $c>0$. Consider the setup in \citet{dehling:taqqu:1989:empirical}:
let the data $\{X_n\}$ be given by the Gaussian subordination model (\ref{eq:subordination}),
where $G$ is any measurable function. We consider a case where the parameter is an infinite-dimensional object: the distribution function $F$. Assume that $F$ is continuous.  Consider the empirical distribution:
\begin{equation}\label{eq:example emp distr}
\widehat{F}_n(x)=\widehat{F}_n(x;\mathbf{X}_1^n)=\frac{1}{n}\sum_{i=1}^n \mathrm{I}\{X_i\le x\}
\end{equation}
as  an estimate of the distribution function $F(x)=P(X_n\le x)$. Let $m$ be the Hermite rank of the class of functions
$$
\mathcal{E}:= \Big\{\mathrm{I}\{G(\cdot)\le x\}-F(x), ~x\in \mathbb{R} \Big\}
 $$
 (see Definition before Theorem 1.1 of \citet{dehling:taqqu:1989:empirical}).
Theorem 1.1 of \citet{dehling:taqqu:1989:empirical} established the following functional limit theorem: if $(2d-1)m>-1$, then we have the weak convergence in the Skorohod space $D(-\infty,\infty)$:
\begin{equation}\label{eq:m est nonnormal emp}
\tau_n\big(\widehat{F}_n(x)-F(x)\big)\Rightarrow
  J_m(x)Z_{m,d}, \quad  \tau_n=n^{(d-1/2)m+1},
\end{equation}
where $Z_{m,d}$ is the Hermite distribution as in (\ref{eq:m est nonnormal}), and $J_m(x)$ is a non-random function determined by the class $\mathcal{E}$. The corresponding result for the short-memory regime $(2d-1)m<-1$ has not been established up to our knowledge (there may be technical issues with tightness, see \citet{chambers:slud:1989:central}), but a weak convergence to a Gaussian process with the rate $\tau_n=n^{1/2}$ is expected (see, e.g. \citet{dehling:philipp:2002:empirical}).

Now consider the block version of (\ref{eq:example emp distr}):
\begin{equation}\label{eq:example block emp distr}
\widehat{F}_{b,k}(x)=\widehat{F}_b(x;\mathbf{X}_{k}^{k+b-1})=\frac{1}{b}\sum_{i=k}^{k+b-1} \mathrm{I}\{X_i\le x\}.
\end{equation}
Again as in the previous sections, we can consider applying subsampling  with a proper estimation of scale $\widehat{\tau}_n$  so that
$$
\widehat{\tau}_n/\tau_n\ConvD W
 $$
 for some nonzero random variable $W$ and
 $$
 \widehat{\tau}_n\big(\widehat{F}_n(x)-F(x)\big)\Rightarrow T(x)
  $$
  for some non-degenerate random  function $T(x)$. Then one can use the empirical ``observations''
$$
\widehat{\tau}_b(\widehat{F}_{b,k}(x)-\widehat{F}_n(x)), \ \ k=1,\ldots,n-b+1
 $$
 for inference. For example, to construct a uniform confidence band for $F$,  consider
$$
S_k=\widehat{\tau}_b \sup_x|\widehat{F}_{b,k}(x)-\widehat{F}_n(x)|,\ k=1,\ldots,n-b+1.
 $$
 Then use the empirical quantile of $\{S_k\}$ to find the cutoff $s_\alpha$ for the confidence band
\begin{equation}\label{eq:confi band}
[\max\{\widehat{F}_n(x)- s_\alpha,0\}, ~\min\{\widehat{F}_n(x)+s_\alpha,1\}],\quad x\in \mathbb{R}.
\end{equation}
 For more information on subsampling empirical processes, see Section 7.4 of \citet{politis:1999:subsampling}.  One needs again to establish results similar to  (\ref{eq:sample cov diff}), (\ref{eq:cov conv normalize}) and  (\ref{eq:self-norm ass}).
 The asymptotic consistency of the confidence band (\ref{eq:confi band}) will then follow under the block size condition $b=b_n\rightarrow\infty$ and $b_n=o(n)$ as $n\rightarrow\infty$.

\section{Proofs of the main results}\label{sec:proofs}
We now give the proofs of the main results stated in Section \ref{sec:main results}. In all the proofs below, the letters $c$, $c_1,c_2,\ldots$ will denote positive constants whose values can change from line to line.
\subsection{Preliminary lemmas}\label{sec:proof lemma}
We give a number of lemmas which will be used in the proofs of the main results. Note that the covariance matrix of the joint vector $\left(\mathbf{Z}_{1}^b, \mathbf{Z}_{k+1}^{k+b}\right)$ is
\[
\begin{pmatrix}
\Sigma_b & \Sigma_{k,b}\\
\Sigma_{k,b}^T & \Sigma_b
\end{pmatrix},
\]
where  $\Sigma_b$ is the in-block covariance matrix  (\ref{eq:Sigma_m}), and $\Sigma_{k,b}$ is the cross-block covariance matrix   (\ref{eq:Sigma_k,b}).
The following lemma states a well-known  relation  between the maximization problem (\ref{eq:rho_k,b}) and the matrices $\Sigma_b$ and $\Sigma_{k,b}$.
\begin{lemma}\label{Lem:can corr matrix}
The supremum in (\ref{eq:rho_k,b}) is attained when $\mathbf{u}=\mathbf{u}^*$ and  $\mathbf{v}=\mathbf{v}^*$, where $\mathbf{u}^*$ is an eigenvector of the matrix
\begin{equation}\label{eq:U matrix}
U_{k,b}=\Sigma_b^{-1}\Sigma_{k,b}\Sigma_b^{-1}\Sigma_{k,b}^T
\end{equation}
 corresponding to its maximum eigenvalue $\lambda_U$, and where $\mathbf{v}^*$ is an eigenvector of the matrix
\begin{equation}\label{eq:V matrix}
V_{k,b}=\Sigma_b^{-1}\Sigma_{k,b}^T\Sigma_b^{-1}\Sigma_{k,b}
\end{equation}
corresponding to its maximum eigenvalue $\lambda_{V}$. In addition, the canonical correlation equals
\[
\rho_{k,b}=\sqrt{\lambda_U}=\sqrt{\lambda_V},
\]
and $\mathbf{u}^*$ and $\mathbf{v}^*$ are related through
\begin{equation}\label{eq:a* b* relation}
\mathbf{u}^*=\rho_{k,b}^{-1}\Sigma_b^{-1}\Sigma_{k,b}\mathbf{v}^*,\qquad \mathbf{v}^*=\rho_{k,b}^{-1}\Sigma_b^{-1}\Sigma_{k,b}^T\mathbf{u}^*.
\end{equation}
\end{lemma}
\begin{proof}
See \citet{hotelling:1936:relations} and also  Section 21.5.3 of \citet{seber:2008:matrix}.
\end{proof}
The following facts about the $\mathrm{FARIMA}(0,d,0)$ model can be found in \citet{brockwell:1991:time}, Section 13.2. See also \citet{hosking:1981:fractional}.
The $\mathrm{FARIMA}(0,d,0)$  time series $\{Z_n\}$ with spectral density $f_d(\lambda)$ in (\ref{eq:f_d}) has covariance function
\begin{equation}\label{eq:gamma_d}
\gamma_d(n)=\gamma_d(0)  \prod_{k=1}^n\frac{k-1+d}{k-d} \sim~ c_d n^{2d-1} \text{ as }n\rightarrow\infty,\qquad  \gamma_d(0)  = \frac{\Gamma(1-2d)}{\Gamma(1-d)^2},
\end{equation}
where $c_d>0$ is a constant, and $\Gamma(\cdot)$ denotes the gamma function defined as
\[
\Gamma(x)=
\int_0^\infty t^{x-1}e^{-t}dt \text{ if }x>0;\quad
\Gamma(x)=+\infty, \text{ if } x=0;\quad
\Gamma(x) = x^{-1}\Gamma(1+x) \text{ if } x<0.
\]
Notice that $\gamma_d(n)>0$ and is decreasing as $n>0$ grows.
It is also known that the mean-square best linear predictor of $Z_{b+1}$
in terms of $Z_1,\ldots,Z_b$ is given as
\begin{equation}\label{eq:predictor 1 step}
\widehat{Z}_{b+1}=P_{[1,b]}Z_{b+1}=\sum_{j=1}^b \phi_{bj}Z_{b-j+1},
\end{equation}
 where $P_{[1,b]}$  denotes the $L^2(\Omega)$ projection onto the closed linear span  $\mathrm{sp}\{Z_1,\ldots,Z_b\}$, and the coefficients are given by
\begin{equation}\label{eq:phi pred}
\phi_{bj}=-\Gamma(-d)^{-1}  { b\choose j} \frac{\Gamma(j-d)\Gamma(b-d-j+1)}{\Gamma(b-d+1)}, \quad j=1,\ldots,b.
\end{equation}
Note also that each $\phi_{bj}>0$ since $\Gamma(x)>0$ if $x>0$ and $\Gamma(x)<0$ if $-1<x<0$. One can then easily deduce the following fact.
\begin{lemma}\label{Lem:positive phi}
Let $\{Z_n\}$ be the $\mathrm{FARIMA}(0,d,0)$ process with spectral  density $f_d$ given in (\ref{eq:f_d}).
Let $\phi_{bj}^n$, $j=1,\ldots,b$, $n\ge b+1$, be the coefficient\footnote{The superscript $n$ in $\phi_{bj}^n$ is an index.} of the best linear predictor of $Z_n$ in terms of $Z_{1},\ldots,Z_{b}$, namely,
$$
P_{[1,b]}Z_n=\sum_{j=1}^b \phi_{bj}^n Z_{b-j+1}.
$$
 Then
 $$
 \phi_{bj}^n>0
 $$
 for $j=1,\ldots,b$ and  $n\ge b+1$.
\end{lemma}
\begin{proof}
Note that
\[
P_{[1,b]}Z_n=P_{[1,b]}\ldots P_{[1,n-2]}P_{[1,n-1]}Z_n.
\]
Then apply (\ref{eq:predictor 1 step}), (\ref{eq:phi pred}) and use the  positiveness of $\phi_{bj}$'s recursively.
\end{proof}

The next result plays a key role in the proof of Theorem \ref{Thm:special}.
\begin{lemma}\label{Lem:postive weights}
If  $\gamma(n)$ is the covariance function of a $\mathrm{FARIMA}(0,d,0)$ time series whose spectral density $f_d$ is given in (\ref{eq:f_d}),
then the matrices $U_{b,k}$ and $V_{b,k}$ in (\ref{eq:U matrix}) and (\ref{eq:V matrix}) respectively have all the entries positive, and the extremal eigenvectors $\mathbf{u}^*$ and $\mathbf{v}^*$ in Lemma \ref{Lem:can corr matrix} can be chosen\footnote{Eigenvectors are determined up to a multiplicative constant.} to have all positive entries.
\end{lemma}
\begin{proof}
To show that the matrices $U_{k,b}$ and $V_{k,b}$ have positive entries, it is enough to show that $\Sigma_b^{-1}\Sigma_{k,b}$ and $\Sigma_b^{-1}\Sigma_{k,b}^T$ have positive entries. Note that a column of $\Sigma_{k,b}$ is of the form
\[
\boldsymbol{\gamma}_{n-1}^{n-b}:=(\gamma(n-1),\ldots,\gamma(n-b))^T
\]
for some $n>b$. The corresponding column of $\Sigma_b^{-1}\Sigma_{k,b}$ is then $\Sigma_b^{-1}\boldsymbol{\gamma}_{n-1}^{n-b}=(\phi_{bb}^n\ldots,\phi_{b1}^n)^T$ by the Yule-Walker equation (see, e.g., (5.1.9) of \citet{brockwell:1991:time}). Similarly, a column of $\Sigma_b^{-1}\Sigma_{k,b}^T$ is of the form $(\phi_{b1}^n,\ldots,\phi_{bb}^n)^T$. Hence by Lemma \ref{Lem:positive phi}, all entries of $\Sigma_b^{-1}\Sigma_{k,b}$ and $\Sigma_b^{-1}\Sigma_{k,b}^T$ are positive.

The Perron-Frobenius Theorem (see Item 9.16 of \citet{seber:2008:matrix}) states that the eigenvector corresponding to the maximum eigenvalue of a matrix with positive  entries can be chosen to have all components positive. Since
the matrix $U_{k,b}$ has positive entries,  we deduce that the extremal eigenvector
$\mathbf{u}^*$ can be chosen to have all positive entries. In view of (\ref{eq:a* b* relation}), this also makes  $\mathbf{v}^*$ positive.
\end{proof}

The following simple fact will be useful.
\begin{lemma}\label{Lem:bound sum a}
Let $\{Z_n\}$ be a $\mathrm{FARIMA}(0,d,0)$ process with spectral density $f_d$ given in (\ref{eq:f_d}). Let $c>0$ be a fixed constant. Then for all nonnegative $u_j$'s  such that
\begin{equation}\label{eq:var constraint}
\Var\left[\sum_{j=1}^b u_j Z_j\right]\le c ,
\end{equation}
 there exists a constant $c_1>0$ which does not depend on $b$, such that
\begin{equation}\label{eq:sum a_m}
\sum_{j=1}^b u_j\le c_1 b^{1/2-d}.
\end{equation}
\end{lemma}
\begin{proof}
Using the fact that $u_j\ge 0$, as well as the positiveness,  monotonicity and the asymptotics of the covariance function $\gamma_d(n)$ in (\ref{eq:gamma_d}), we have
\[
c\ge \Var\left[\sum_{j=1}^b u_j Z_j\right]=\sum_{i,j=1}^b u_iu_j \gamma_d(i-j)\ge\gamma_d(b-1) \sum_{i,j=1}^b u_i u_j\ge c_2 b^{2d-1} \left(\sum_{j=1}^b u_j\right)^2
\]
for some constant $c_2>0$ not depending on $b$, which yields (\ref{eq:sum a_m}).
\end{proof}
\begin{remark}
By making use of the result of \citet{adenstedt:1974:large}, one can  strengthen (\ref{eq:sum a_m}) to  $\left|\sum_{j=1}^b u_j\right|\le c b^{1/2-d}$  with $u_j$'s not necessarily nonnegative. See Lemma 2 of \citet{betken:Wendler:2015:subsampling}.
\end{remark}
We  recall here the time and frequency domain  isomorphism (Kolmogorov Isomorphism).
\begin{lemma}[\citet{brockwell:1991:time}, Theorem 4.8.1]\label{Lem:iso}
Suppose $f$ is the spectral density of $\{Z_n\}$.
 Let $\mathcal{H}=\overline{\mathrm{sp}}\{Z_n, ~n\in \mathbb{Z}\}$
be the Hilbert space spanned by $\{Z_n\}$ in $L^2(\Omega)$. Let
 $\mathfrak{H}=\overline{\mathrm{sp}}\{e^{in\cdot},~ n\in \mathbb{Z}\}$ be the Hilbert space spanned by $\{e^{in\cdot}\}$ in $L^2((-\pi,\pi],\mathbb{C};f)$ ($f$-weighted  complex-valued $L^2$ space on $(-\pi,\pi]$). Then there is a unique Hilbert space  isomorphism
\[
T:  \mathcal{H} \longrightarrow   \mathfrak{H} ,\quad Z_n \longrightarrow e^{in\cdot}.
\]
\end{lemma}
We note that the translation in the time domain of $k$ units by the isomorphism acts as multiplication by $e^{ik\lambda}$  in the frequency domain.

Let us return to the maximization problem (\ref{eq:rho_k,b}). First note that
\[
\sup_{\mathbf{u}\in \mathbb{R}^b, \mathbf{v}\in \mathbb{R}^b} \Corr\Big(\langle\mathbf{u},\mathbf{Z}_1^b\rangle, \langle\mathbf{v}, \mathbf{Z}_{k+1}^{k+b}\rangle\Big)=\sup_{\mathbf{u}\in \mathbb{R}^b, \mathbf{v}\in \mathbb{R}^b} \Big|\Corr\Big(\langle\mathbf{u},\mathbf{Z}_1^b\rangle, \langle\mathbf{v}, \mathbf{Z}_{k+1}^{k+b}\rangle\Big)\Big|.
\]
Note also  that
the preceding maximization   can be stated in an equivalent constrained  optimization form:
\begin{align}\label{eq:const opt time}
&\rho_{k,b}=\sup_{\mathbf{u}\in \mathbb{R}^b, \mathbf{v}\in \mathbb{R}^b} \Big|\Cov\Big(\langle\mathbf{u},\mathbf{Z}_1^b\rangle, \langle\mathbf{v}, \mathbf{Z}_{k+1}^{k+b}\rangle\Big)\Big|=\sup_{\mathbf{u}\in \mathbb{R}^b, \mathbf{v}\in \mathbb{R}^b}\Big|\sum_{i,j=1}^b u_iv_j \gamma(k+j-i)\Big| ,\notag\\
&\text{Subject to: }   \Var\left[\langle\mathbf{u},\mathbf{Z}_1^b\rangle\right]=\sum_{i,j=1}^b u_iu_j\gamma(i-j)\le 1,\quad \Var\left[\langle\mathbf{v},\mathbf{Z}_{k+1}^{k+b}\rangle\right]=\sum_{i,j=1}^b v_iv_j\gamma(i-j)\le 1,
\end{align}
that is, the conditions
 $$
 \Var\left[\langle\mathbf{v},\mathbf{Z}_{1}^{b}\rangle\right]=1\quad \mbox{\rm and}\quad \Var\left[\langle\mathbf{v},\mathbf{Z}_{k+1}^{k+b}\rangle\right]=1$$
  can be replaced by
  $$
  \Var\left[\langle\mathbf{v},\mathbf{Z}_{1}^{b}\rangle\right]\le 1 \quad \mbox{\rm and}\quad  \Var\left[\langle\mathbf{v},\mathbf{Z}_{k+1}^{k+b}\rangle\right]\le 1.
  $$
This is because the maximum will be attained at the boundaries where the variances are equal to $1$ by scaling.
In view of Lemma \ref{Lem:iso}, the preceding constrained optimization can be expressed in the frequency-domain as
\begin{align}
&\rho_{k,b}=\sup_{U_b,V_b} \Big|\int_{-\pi}^\pi e^{ik\lambda} U_b(e^{i\lambda})\overline{V_b(e^{i\lambda})} f(\lambda) d\lambda\Big|=\sup_{U_b,V_b}\Big|\int_{-\pi}^\pi e^{-ik\lambda} U_b(e^{i\lambda})\overline{V_b(e^{i\lambda})} f(\lambda) d\lambda\Big|\label{eq:opt freq}\\
&\text{Subject to: }  \int_{-\pi}^\pi   |U_b(e^{i\lambda})|^2 f(\lambda)d\lambda \le 1,\quad  \int_{-\pi}^\pi   |V_b(e^{i\lambda})|^2 f(\lambda)d\lambda \le 1, \label{eq:const opt freq}
\end{align}
where the supremum is taken over all  polynomials $U_b(z)=\sum_{j=0}^{b-1} u_{j+1} z^j $, $V_b(z)=\sum_{j=0}^{b-1} v_{j+1} z^j$.
\begin{remark}\label{Rem:c replace 1}
If one replaces the constraints ``$\ldots\le 1$'' in (\ref{eq:const opt time}) or (\ref{eq:const opt freq}) by ``$\ldots \le c$'' for some constant $c>0$, then the supremum obtained in (\ref{eq:const opt time}) and (\ref{eq:const opt freq}) becomes $c\rho_{k,b}$.
\end{remark}

\subsection{Proof of the main theorems}\label{sec:pf thm}
\begin{proof}[Proof of Theorem \ref{Thm:special}]
We first use  the maximizer  $\mathbf{u}^*=(u_1^*,\ldots,u_b^*)^T$ and $\mathbf{v}^*=(v_1^*,\ldots,v_b^*)^T$ of (\ref{eq:const opt time}) to get
\[
\rho_{k,b}=\sum_{i,j=1}^b u_i^* v_j^*\gamma_d(k+j-i).
\]
By Lemma \ref{Lem:postive weights}, the components   $u_j^*$ and $v_j^*$, $j=1,\ldots,b$, can be chosen to be all positive. In view of (\ref{eq:const opt time}), we can suppose (\ref{eq:var constraint}) with $c=1$. By the positiveness and the monotone decreasing property of $\gamma_d(n)$ in (\ref{eq:gamma_d}), as well as the relation (\ref{eq:sum a_m}) in Lemma \ref{Lem:bound sum a}, we have for $k>b$,
\begin{align*}
\rho_{k,b}\le \left(\sum_{j=1}^b u_j^*\right)\left(\sum_{j=1}^b v_j^* \right) \gamma_d(k+1-b)
 &\le c  b^{1/2-d}b^{1/2-d} (k-b)^{2d-1}=c b^{1-2d}(k-b)^{2d-1}.
\end{align*}
\end{proof}

The following corollary will be used in the proof of Theorem \ref{Thm:main}. It is an immediate consequence of  Theorem \ref{Thm:special}, the  frequency-domain characterization (\ref{eq:const opt freq}) and
Remark \ref{Rem:c replace 1}.
\begin{corollary}\label{Cor:g_b}
Define $g_b(\lambda)=U_b(e^{i\lambda}) \overline{V_b(e^{i\lambda})}f_d(\lambda)$, where $f_d$ is as in (\ref{eq:f_d}), and let
\[
\widehat{g}_b(k)=\int_{-\pi}^{\pi} e^{ik\lambda} g_b(\lambda)d\lambda, \qquad  k\in \mathbb{Z},
\]
be its Fourier coefficient. Then under the constraints
 $$\int_{-\pi}^\pi   |U_b(e^{i\lambda})|^2 f_d(\lambda)d\lambda \le c_1\ \  \mbox{\it and}\ \  \int_{-\pi}^\pi   |V_b(e^{i\lambda})|^2 f_d(\lambda)d\lambda \le c_1,
  $$
  where $c_1>0$ is a constant, we have for some constant $c>0$ and $1\le b<k$ that
\[
\sup_{U_b,V_b} \left|\widehat{g}_b(k)\right|=c_1\rho_{k,b} \le c b^{1-2d}(k-b)^{2d-1},
\]
where $\rho_{k,b}$ is the canonical correlation corresponding to the spectral density $f_d(\lambda)$.
\end{corollary}

We are now ready to prove Theorem \ref{Thm:main}.
\begin{proof}[Proof of Theorem \ref{Thm:main}]
Because $f=f_df_0$ and $f_0\ge c_0>0$ by assumption, the constraint in (\ref{eq:const opt freq}) implies the following constraint:
\begin{equation}\label{eq:new const}
\int_{-\pi}^\pi   |U_b(e^{i\lambda})|^2 f_d(\lambda)d\lambda \le c_0^{-1}\ \ \mbox{\it and} \ \
\int_{-\pi}^\pi   |V_b(e^{i\lambda})|^2 f_d(\lambda)d\lambda \le c_0^{-1}.
\end{equation}
Using the notation in (\ref{eq:const opt freq}), set as in Corollary \ref{Cor:g_b}
\begin{equation}\label{eq:g_b}
g_b(\lambda)=U_b(e^{i\lambda})\overline{V_b(e^{i\lambda})} f_d(\lambda),
\end{equation}
which involves the spectral density $f_d$.
The first integral in (\ref{eq:opt freq}) becomes
\begin{align*}
\int_{-\pi}^{\pi} e^{ik\lambda} g_b(\lambda)f_0(\lambda)d\lambda.
\end{align*}
We have (see Corollary 4.3.2 of \citet{brockwell:1991:time})
\[
f_0(\lambda)= \frac{1}{2\pi}\sum_{n }e^{-in\lambda} \gamma_0(n)
,\]
and
$f_0(\lambda)\le(2\pi)^{-1}\sum_n|\gamma_0(n)|<\infty$. Note also that $\int_{-\pi}^\pi |g_b(\lambda)| d\lambda <\infty $ since $U_b$ and $V_b$ are polynomials which are bounded on $(-\pi,\pi]$. So one gets by Fubini's theorem that
\begin{align}
\int_{-\pi}^{\pi} e^{ik\lambda} g_b(\lambda)f_0(\lambda)d\lambda=&\frac{1}{2\pi}\int_{-\pi}^{\pi} e^{ik\lambda} g_b(\lambda)\left(\sum_{n}\gamma_0(n)e^{-in\lambda}\right) d\lambda\notag\\=&\sum_{n} \widehat{g}_b(k-n)\gamma_0(n)=\sum_{n} \widehat{g}_b(n)\gamma_0(k-n).\label{eq:convolution}
\end{align}

In view of (\ref{eq:opt freq}) and (\ref{eq:convolution}), one has
\begin{equation}\label{eq:to be split}
\rho_{k,b}=\sup_{U_b,V_b} \left|\int_{-\pi}^{\pi} e^{ik\lambda} g_b(\lambda)f_0(\lambda)d\lambda\right|=\sup_{U_b,V_b} \left|\sum_{n} \widehat{g}_b(n)\gamma_0(k-n)\right|,
\end{equation}
where $\rho_{k,b}$ is the canonical correlation corresponding to the spectral density $f=f_df_0$. We will now split the last expression in (\ref{eq:to be split}) into two terms, one involving $|n|>k'$ and other  $|n|\le k'$, where $b<k'\le k(1-\epsilon)$ as in (\ref{eq:m k' restrict}).
Hence
\begin{equation}\label{eq:rho_k,b bound in pf}
\rho_{k,b}\le    \sup_{U_b,V_b}\sum_{|n|> k'} |\widehat{g}_b(n)\gamma_0(k-n)|+\sup_{U_b,V_b} \sum_{|n|\le k'} |\widehat{g}_b(n)\gamma_0(k-n)|=:T_1+T_2.
\end{equation}

The first term can be bounded as
\begin{align*}
T_1\le  \left(\sum_{s=-\infty}^\infty|\gamma_0(s)|\right) \max_{|n|>k'} \sup_{U_b,V_b}|\widehat{g}_b(n)|,
\end{align*}
where
\begin{equation*}
\sup_{U_b,V_b}|\widehat{g}_b(n)|=\sup_{U_b,V_b}\left |\int_{-\pi}^\pi e^{in\lambda} U_b(e^{i\lambda})\overline{V_b(e^{i\lambda})} f_d(\lambda) d\lambda \right|=\sup_{U_b,V_b}\left |\int_{-\pi}^\pi e^{-in\lambda} U_b(e^{i\lambda})\overline{V_b(e^{i\lambda})} f_d(\lambda) d\lambda \right|.
\end{equation*}
In view of (\ref{eq:new const}), Corollary \ref{Cor:g_b} with $c_1=c_0^{-1}$ yields that
\[
\sup_{U_b,V_b}|\widehat{g}_b(n)|\le c (n-b)^{2d-1}b^{1-2d}.
\]
Hence
\begin{equation}\label{eq:T_1}
T_1\le   \left(\sum_{s=-\infty}^\infty|\gamma_0(s)|\right) c \max_{n>k'} (n-b)^{2d-1}b^{1-2d}\le c_1 (k'-b)^{2d-1}b^{1-2d}.
\end{equation}

We now deal with $T_2$. First note that by (\ref{eq:g_b}),  the Cauchy-Schwartz inequality and (\ref{eq:new const}),
one has
\begin{equation}\label{eq:bound g_b hat}
|\widehat{g}_b(n)|\le \int_{-\pi}^\pi  |U_b(e^{i\lambda})\overline{V_b(e^{i\lambda})}| f_d(\lambda)d\lambda   \le c_0^{-1}.
\end{equation}
$\bullet$ If the assumption $\gamma_0(n)=O(n^{-\alpha})$ holds, then (\ref{eq:bound g_b hat}) and  the restriction $k'\le k(1-\epsilon)$ imply that
\begin{equation}\label{eq:T_2}
T_2\le  c_0^{-1}\sum_{|n|\le k'} |\gamma_0(k-n)|\le c\sum_{|n|\le k'} (k-n)^{-\alpha}\le c\sum_{|n|\le k'}  (k-k')^{-\alpha} \le c_1 k'  (k-k')^{-\alpha}\le c_1 \epsilon^{-\alpha} k'k^{-\alpha}.
\end{equation}
$\bullet$ If instead  $\gamma_0(n)=o(n^{-\alpha})$ is assumed, then
\begin{equation}\label{eq:T_2 o}
T_2\le  c_0^{-1}\sum_{|n|\le k'} |\gamma_0(k-n)|\le c k'  o\left((k-k')^{-\alpha}\right)\le c \epsilon^{-\alpha} k'o(k^{-\alpha}).
\end{equation}
$\bullet$ If alternatively  $\gamma_n(n)=O(e^{-cn})$ is assumed,  then
\begin{equation}\label{eq:T_2 alt}
T_2\le  c_0^{-1}\sum_{n=-\infty}^{k'} |\gamma_0(k-n)|\le c_1   e^{-c(k-k')}\le  c_1 e^{-\epsilon c k}.
\end{equation}
Combining (\ref{eq:rho_k,b bound in pf}), (\ref{eq:T_1}), (\ref{eq:T_2}), (\ref{eq:T_2 o}) and (\ref{eq:T_2 alt}) yields the desired bounds.
\end{proof}

\begin{lemma}\label{Lem:reduction max}
\footnote{This argument and the one in the following theorem were suggested by an anonymous referee.}
Suppose that for any  $\epsilon\in (0,1)$,  as $n\rightarrow\infty$ and $b_n=o(n)$, we have
\begin{equation}\label{eq:max alpha tend zero}
\max_{[\epsilon n]\le k\le n}\rho_{k,b_n}\rightarrow 0.
\end{equation}
Then $\sum_{k=1}^n\rho_{k,b_n}=o(n)$ as $n\rightarrow\infty$ . Furthermore, if $\rho_{k,b}$ is non-increasing in $k$ for each $b$, then the converse holds.
\end{lemma}
\begin{proof}
For an arbitrarily small $\epsilon>0$. Indeed, if (\ref{eq:max alpha tend zero}) holds, then using $\rho_{k,b}\le 1$, one has
\[
\sum_{k=1}^n \rho_{k,b_n}=\sum_{k<[\epsilon n]} \rho_{k,b_n}
+\sum_{[\epsilon n]\le k\le n}\rho_{k,b_n}\le \epsilon n+ \sum_{[\epsilon n]\le k\le n}\max_{[\epsilon n]\le k\le n}\rho_{k,b_n}.
\]
By (\ref{eq:max alpha tend zero}),
\[
\limsup_{n\rightarrow\infty} \frac{1}{n}\sum_{k=1}^n \rho_{k,b_n} \le \epsilon.
\]
By the arbitrariness of $\epsilon$, we get $\sum_{k=1}^n \rho_{k,b_n}=o(n)$.

Now suppose that $\rho_{k,b}$ is non-increasing in $k$. Assume $\frac{1}{n}\sum_{k=1}^n \rho_{k,b_n}\rightarrow0$ as $n\rightarrow\infty$. Then for any $\epsilon\in (0,1)$, we have
\[
\frac{1}{n}\sum_{k=1}^n \rho_{k,b_n}\ge \frac{1}{n}\sum_{k=1}^{[\epsilon n]}\rho_{k,b_n}\ge  \frac{1}{n}\sum_{k=1}^{[\epsilon n]}\rho_{[\epsilon n],b_n}=\frac{[n\epsilon]}{n} \rho_{[\epsilon n],b_n}.
\]
Taking $n\rightarrow\infty$ on both sides, we get  $\rho_{[\epsilon n],b_n}=\max_{[\epsilon n]\le k\le n}\rho_{k,b_n}\rightarrow 0$.
\end{proof}

\begin{proof}[Proof of Theorem \ref{Thm:block cond}]

By Lemma \ref{Lem:reduction max},  we need to prove (\ref{eq:max alpha tend zero}). First fix an integer $m\ge 1$ and choose
\begin{equation}\label{eq:k' choice}
 k'=b_n(m+1).
\end{equation}Even though $m$ can be quite large, $k'=b_n(m+1)=o(n)$ because
 $b_n=o(n)$. In fact, for $n$ large enough, we have the following chain of inequalities:

 $$
 1\le b_n< k'=b_n(m+1)\le k/2 \quad \textrm{for} \quad
  [\epsilon n]\le k \le n.
 $$
 The inequalities $1\le b_n< k'\le k/2$
 allow the application of
  Theorem \ref{Thm:main} (a), yielding
\begin{equation}\label{eq:rho_k,b bound special}
\rho_{k,b_n}\le c_1 \left(\frac{b_n}{k'-b_n}\right)^{1-2d} + c_2 \frac{k'}{k}.
\end{equation}
Then by $k'=b_n(m+1)$, (\ref{eq:rho_k,b bound special}) and $b_n=o(n)$,
\[
\limsup_{n\rightarrow\infty} \max_{[n\epsilon]\le k\le n}\rho_{k,b_n}\le  c_1 m^{2d-1}+ c_2(m+1)\limsup_{n\rightarrow\infty} b_n [\epsilon n]^{-1}=c_1m^{2d-1}.
\]
Since $m$ can be chosen arbitrarily large and $2d-1<0$, we get (\ref{eq:max alpha tend zero}).

\end{proof}
\appendix
\section*{Appendix}
To obtain Proposition \ref{Pro:bw} and get a bound on $\rho_{k,b}$  in (\ref{eq:rho_k,b}),
\citet{betken:Wendler:2015:subsampling} imposed  the following assumptions on the  covariance function  $\gamma(n)$ and the spectral density  $f(\lambda)$ of $\{Z_n\}$.
\begin{enumerate}[\textit{BW}1.]
\item  The covariance function satisfies $\gamma(n)=n^{2d-1}L_\gamma(n)$, $d\in (0,1/2)$, where $L_\gamma(n)$ is a function slowly varying as $n\rightarrow\infty$, which satisfies $\max_{n+1\le k\le n+2m-1}|L_\gamma(k)-L_\gamma(n)|\le C (m/n) \min\{L_\gamma(n),1\}$ for some constant $C>0$ and all $n,m\in \mathbb{Z}_+$.
\item  The spectral density $f(\lambda)=|\lambda|^{-2d}L_f(\lambda)$, where $L_f(\lambda)$ is a slowly varying function as $\lambda\rightarrow 0+$ and satisfies  $L_f(\lambda)\ge c$ for some $c>0$ and the limit of $L_f(\lambda)$ exists (can be $+\infty$) as $\lambda\rightarrow 0$\footnote{In \citet{betken:Wendler:2015:subsampling}, the existence of the limit $L_f(\lambda)$ as $\lambda\rightarrow 0$ is  not assumed. However, this seems necessary,  since the result in \citet{adenstedt:1974:large}  applied by \citet{betken:Wendler:2015:subsampling} requires  continuity at $\lambda=0$ of $1/L(\lambda)$ (see the proof of Lemma 2 in \citet{betken:Wendler:2015:subsampling}, which quoted Lemma 4.4 of \citet{adenstedt:1974:large}).}.
\end{enumerate}

\medskip
\noindent\textbf{Acknowledgments.} We would like to thank Mamikon S.\ Ginovyan and Ting Zhang for useful discussions. We also thank the referees for their helpful comments    which lead to significant improvement of the paper. This work was partially supported by the NSF grant  DMS-1309009  at Boston University.

\bibliographystyle{plainnat}
\bibliography{Bib}

\end{document}